\documentclass[12pt]{amsart}

\usepackage{amsmath,amsfonts,amssymb}
\usepackage{graphicx}
\usepackage[numbers,square,sort&compress]{natbib}
\usepackage[colorinlistoftodos]{todonotes}
\usepackage{times}
\usepackage{xspace}
\usepackage{xypic}
\usepackage{color}
\usepackage{a4wide}
\usepackage{tikz}
\usepackage{subfig}
\usepackage{enumitem}

\graphicspath{{./}{./figs/}{../../figs/}}

\theoremstyle{definition}
\newtheorem{definition}{Definition}[section]
\newtheorem{prop}[definition]{Proposition}
\newtheorem{lemma}[definition]{Lemma}
\newtheorem{theorem}[definition]{Theorem}
\newtheorem{cor}[definition]{Corollary}
\newtheorem{remark}[definition]{Remark}

\newtheorem{example}[definition]{Example}

\def\N{\mathbb{N}}
\def\R{\mathbb{R}}

\def\tsimkappa{\sim_{\kappa}}

\DeclareMathOperator{\Acyc}{Acyc}

\DeclareMathOperator{\C}{C}
\DeclareMathOperator{\FC}{FC}
\DeclareMathOperator{\CFC}{CFC}

\DeclareMathOperator{\GL}{GL}

\DeclareMathOperator{\Fib}{Fib}
\DeclareMathOperator{\supp}{supp}

\def\<{\langle}
\def\>{\rangle}

\newcommand{\w}{\mathsf{w}}

\def\qed{$\hfill\Box$}
\def\e{\vec{e}}
\def\r{\vec{r}}
\def\u{\mathsf{u}}

\def\w{\mathsf{w}}
\def\x{\mathbf{x}}
\def\y{\mathbf{y}}
\def\z{\vec{z}}
\def\0{\vec{0}}

\newcommand{\smc}[1]{\textrm{\textsc{#1}}}
\newcommand{\sym}{\smc{Sym}}
\newcommand{\COMMENT}[1]{}

\newcounter{comment}

\begin{document}

\title[CFC Elements of Coxeter Groups]{On the Cyclically Fully
  Commutative Elements of Coxeter Groups}
  
\author{T.~Boothby} \address{Department of Mathematics \\ Simon Fraser
  University \\ Burnaby, BC V5A 1S6} \email{tboothby@sfa.ca}
\author{J.~Burkert} \address{Department of Mathematics \\ Harvey Mudd
  College \\ Claremont, CA 91711} \email{jeffrey.burkert@gmail.com}
\author{M.~Eichwald} \address{Department of Mathematical Sciences
  \\ University of Montana \\ Missoula, MT 59812}
\email{morgan.eichwald@gmail.com} \author{D.C.~Ernst}
\address{Mathematics Department \\ Plymouth State University
  \\ Plymouth, NH 03264} \email{dcernst@plymouth.edu}
\author{R.M. Green} \address{Department of Mathematics \\ University
  of Colorado \\ Boulder, CO 80309} \email{rmg@euclid.colorado.edu}
\author{M. Macauley} \address{Department of Mathematical Sciences
  \\ Clemson University \\ Clemson, SC 29634}
\email{macaule@clemson.edu}

\thanks{Four of us (T.B., J.B., M.E., M.M.) gratefully acknowledge
  support from National Science Foundation Grant DMS-0754486, and from
  the University of Washington.}



\begin{abstract}
  Let $W$ be an arbitrary Coxeter group. If two elements have
  expressions that are cyclic shifts of each other (as words), then
  they are conjugate (as group elements) in $W$. We say that $w$ is
  \emph{cyclically fully commutative} (CFC) if every cyclic shift of
  any reduced expression for $w$ is fully commutative (i.e., avoids
  long braid relations). These generalize Coxeter elements in that
  their reduced expressions can be described combinatorially by
  acyclic directed graphs, and cyclically shifting corresponds to
  source-to-sink conversions. In this paper, we explore the
  combinatorics of the CFC elements and enumerate them in all Coxeter
  groups. Additionally, we characterize precisely which CFC elements
  have the property that powers of them remain fully commutative,
  via the presence of a simple combinatorial feature called a
  \emph{band}. This allows us to give necessary and sufficient
  conditions for a CFC element $w$ to be \emph{logarithmic}, that is,
  $\ell(w^k)=k\cdot \ell(w)$ for all $k\geq 1$, for a large class of
  Coxeter groups that includes all affine Weyl groups and simply-laced
  Coxeter groups. Finally, we give a simple non-CFC element that fails
  to be logarithmic under these conditions.
\end{abstract}

\maketitle


\section{Introduction}

A classic result of Coxeter groups, known as Matsumoto's
theorem~\cite{Matsumoto:64}, states that any two reduced expressions
of the same element differ by a sequence of braid relations. If two
elements have expressions that are cyclic shifts of each other (as
words), then they are conjugate (as group elements). We say that an
expression is cyclically reduced if every cyclic shift of it is
reduced, and ask the following question, where an affirmative answer
would be a ``cyclic version'' of Matsumoto's theorem.
\begin{quote}
\centering \emph{Do two cyclically reduced expressions of conjugate
  elements differ by a sequence of braid relations and cyclic shifts}?
\end{quote}
While the answer to this question is, in general, ``no,'' it seems to
``often be true,'' and understanding when the answer is ``yes'' is a
central focus of a broad ongoing research project of the last three
authors. It was recently shown to hold for all Coxeter
elements~\cite{Speyer:09,Eriksson:09}, though the result was not
stated in this manner. Key to this was establishing necessary and
sufficient conditions for a Coxeter element $w\in W$ to be
\emph{logarithmic}, that is, for $\ell(w^k)=k\cdot \ell(w)$ to hold
for all $k\geq 1$. Trying to understand which elements in a Coxeter
group are logarithmic motivated this work. Here, we introduce and
study a class of elements that generalize the Coxeter elements, in
that they share certain key combinatorial properties.

A Coxeter element is a special case of a \emph{fully commutative} (FC)
element~\cite{Stembridge:96}, which is any element with the property
that any two reduced expressions are equivalent by only short braid
relations (i.e., iterated commutations of commuting generators). In
this paper, we introduce the \emph{cyclically fully commutative} (CFC)
elements. These are the elements for which every cyclic shift of any
reduced expression is a reduced expression of an FC element. If we
write a reduced expression for a cyclically reduced element in a
circle, thereby allowing braid relations to ``wrap around the end of
the word,'' the CFC elements are those where only short braid
relations can be applied. In this light, the CFC elements are the
``cyclic version'' of the FC elements. In particular, the cyclic
version of Matsumoto's theorem for the CFC elements asks when two
reduced expressions for conjugate elements $w$ and $w'$ are equivalent
via only short braid relations and cyclic shifts. As with Coxeter
elements, the first step in attacking this problem is to find
necessary and sufficient conditions for a CFC element to be
logarithmic.

This paper is organized as follows. After necessary background
material on Coxeter groups is presented in
Section~\ref{sec:coxetergroups}, we introduce the CFC elements in
Section~\ref{sec:cfc}. We motivate them as a natural generalization of
Coxeter elements, in the sense that like Coxeter elements, they can be
associated with canonical acyclic directed graphs, and a cyclic shift
(i.e., conjugation by a generator) of a reduced expression corresponds
on the graph level to converting a source into a sink. In
Section~\ref{sec:cfc-properties}, we prove a number of combinatorial
properties of CFC elements, and introduce the concept of a
\emph{band}, which tells us precisely when powers of a CFC element
remain fully commutative (Theorem~\ref{thm:bands}). In
Section~\ref{sec:cfc-enumeration}, we enumerate the CFC elements in
all Coxeter groups, and we give a complete characterization of the CFC
elements in groups that contain only finitely many. In
Section~\ref{sec:rootaut}, we formalize the root automaton of a
Coxeter group in a new way. We then use it to prove a new result on
reducibility, which we utilize in Section~\ref{sec:affineweyl} to
establish necessary and sufficient conditions for CFC elements to be
logarithmic, as long as they have no ``large bands''
(Theorem~\ref{thm:largebands}). We conclude that in any Coxeter group
without ``large odd endpoints'' (a class of groups includes all affine
Weyl groups and simply-laced Coxeter groups) a CFC element is
logarithmic if and only if it is \emph{torsion-free}
(Corollary~\ref{cor:baby-RHD}). The CFC assumption is indeed crucial
for being logarithmic; as we conclude with a simple counterexample in
$\widetilde{C}_2$ by dropping only the CFC condition.

\section{Coxeter groups}\label{sec:coxetergroups}

A \emph{Coxeter group} is a group $W$ with a distinguished set of
generating involutions $S$ with presentation
\[
\<s_1,\dots,s_n\mid (s_is_j)^{m_{i,j}}=1\>\,,
\]
where $m_{i,j}:=m(s_{i},s_{j})=1$ if and only if $s_i=s_j$.  The
exponents $m(s,t)$ are called \emph{bond strengths}, and it is
well-known that $m(s,t)=|st|$. We define $m(s,t)$ to be $\infty$ if
there is no exponent $k>0$ such that $(st)^{k}=1$.  A Coxeter group is
\emph{simply-laced} if each $m(s,t)\leq 3$. If $S=\{s_1,\dots,s_n\}$,
the pair $(W,S)$ is called a \emph{Coxeter system} of \emph{rank}
$n$. A Coxeter system can be encoded by a unique \emph{Coxeter graph}
$\Gamma$ having vertex set $S$ and edges $\{s,t\}$ for each
$m(s,t)\geq 3$. Moreover, each edge is labeled with its corresponding
bond strength, although typically the labels of $3$ are omitted
because they are the most common. If $\Gamma$ is connected, then $W$
is called \emph{irreducible}.

Let $S^*$ denote the free monoid over $S$. If a word
$\w=s_{x_1}s_{x_2}\cdots s_{x_m}\in S^*$ is equal to $w$ when
considered as an element of $W$, we say that $\w$ is an
\emph{expression} for $w$. (Expressions will be written in {\sf sans
  serif} font for clarity.) If furthermore, $m$ is minimal, we say
that $\w$ is a \emph{reduced expression} for $w$, and we call $m$ the
\emph{length} of $w$, denoted $\ell(w)$. If every cyclic shift of $\w$
is a reduced expression for some element in $W$, then we say that $\w$
is \emph{cyclically reduced}. A group element $w\in W$ is cyclically
reduced if every reduced expression for $w$ is cyclically reduced.

The \emph{left descent set} of $w\in W$ is the set $D_L(w)=\{s\in
S\mid \ell(sw)<\ell(w)\}$, and the \emph{right descent set} is defined
analogously as $D_R(w)=\{s\in S\mid \ell(ws)<\ell(w)\}$. If $s\in
D_L(w)$ (respectively, $D_R(w)$), then $s$ is said to be
\emph{initial} (respectively, \emph{terminal}). It is well-known that
if $s\in S$, then $\ell(sw)=\ell(w)\pm 1$, and so $\ell(w^k)\leq
k\cdot \ell(w)$. If equality holds for all $k\in\N$, we say that $w$ is
\emph{logarithmic}.

For each integer $m\geq 0$ and distinct generators $s,t\in S$, define
\[
\<s,t\>_m=\underbrace{stst\cdots}_{m}\in S^*\,.
\]
The relation $\<s,t\>_{m(s,t)}=\<t,s\>_{m(s,t)}$ is called a
\emph{braid relation}, and is additionally called a \emph{short braid
  relation} if $m(s,t)=2$. (Some authors call
$\<s,t\>_{m(s,t)}=\<t,s\>_{m(s,t)}$ a short braid relation if
$m(s,t)=3$, and a commutation relation if $m(s,t)=2$.) The short braid
relations generate an equivalence relation on $S^*$, and the resulting
equivalence classes are called \emph{commutation classes}. If two
reduced expressions are in the same commutation class, we say they are
\emph{commutation equivalent}.  An element $w\in W$ is \emph{fully
  commutative} (FC) if all of its reduced expressions are commutation
equivalent, and we denote the set of FC elements by $\FC(W)$.  For
consistency, we say that an expression $\w \in S^*$ is FC if it is a
reduced expression for some $w\in\FC(W)$. If $\w$ is not FC, then it
is commutation equivalent to a word $\w'$ for which either $ss$ or
$\<s,t\>_{m(s,t)}$ appears as a consecutive subword, with $m(s,t)\geq
3$ (this is not immediately obvious; see
Proposition~\ref{prop:badchain}).

The braid relations generate a coarser equivalence relation on
$S^*$. Matsumoto's theorem~\cite[Theorem 1.2.2]{Geck:00} says that an
equivalence class containing a reduced expression must consist
entirely of reduced expressions, and that the set of all such
equivalence classes under this coarser relation is in 1--1
correspondence with the elements of $W$.
\begin{theorem}[Matsumoto's theorem]
  In a Coxeter group $W$, any two reduced expressions for the same
  group element differ by braid relations. \hfill\qed
\end{theorem}
Now, consider an additional equivalence relation $\tsimkappa$,
generated by cyclic shifts of words, i.e.,
\begin{equation}
  \label{eq:shift}
  s_{x_1}s_{x_2}\cdots s_{x_m}\longmapsto s_{x_2}s_{x_3}\cdots
  s_{x_m}s_{x_1}\,.
\end{equation}
The resulting equivalence classes were studied in~\cite{Macauley:11c}
and are in general, finer than conjugacy classes, but they often
coincide. Determining conditions for when $\kappa$-equivalence and
conjugacy agree would lead to a ``cyclic version'' of Matsumoto's
theorem for some class of elements, and is one of the long-term
research goals of the last three authors.
\begin{definition}
  Let $W$ be a Coxeter group. We say that a conjugacy class $C$
  satisfies the \emph{cyclic version of Matsumoto's theorem} if any two
  cyclically reduced expressions of elements in $C$ differ by braid relations
  and cyclic shifts.
\end{definition}
One only needs to look at type $A_n$ (the symmetric group
$\sym_{n+1}$) to find an example of where the cyclic version of
Matsumoto's theorem fails. Any two simple generators in $A_n$ are
conjugate, e.g., $s_1s_2(s_1)s_2s_1=s_2$. However, for longer words,
such examples appear to be less common, and we would like to
characterize them.

The \emph{support} of an expression $\w\in S^*$ is simply the set of
generators that appear in it. As a consequence of Matsumoto's theorem,
it is also well-defined to speak of the support of a group element
$w\in W$, as the set of generators appearing in any reduced expression
for $w$. We denote this set by $\supp(w)$, and let $W_{\supp(w)}$ be
the (standard parabolic) subgroup of $W$ that it generates. If
$W_{\supp(w)}=W$ (i.e., $\supp(w)=S$), we say that $w$ has \emph{full
  support}. If $W_{\supp(w)}$ has no finite factors, or equivalently,
if every connected component of $\Gamma_{\supp(w)}$ (i.e., the
subgraph of $\Gamma$ induced by the support of $w$) describes an
infinite Coxeter group, then we say that $w$ is
\emph{torsion-free}. The following result is straightforward.
\begin{prop}\label{prop:logarithmic}
  Let $W$ be a Coxeter group. If $w\in W$ is logarithmic, then $w$ is
  cyclically reduced and torsion-free.
\end{prop}
\begin{proof}
  If $w$ is not cyclically reduced, then there exists a sequence of
  cyclic shifts of some reduced expression of $w$ that results in a
  non-reduced expression.  In this case, there exists $w_{1}, w_{2}\in
  W$ such that $w=w_{1}w_{2}$ (reduced) while $\ell(w_{2}w_{1})<
  \ell(w)$.  This implies that
  \[
  \ell(w^2)=\ell(w_1w_2w_1w_2)\leq\ell(w_1)+\ell(w_2w_1)+\ell(w_2)<2\ell(w)\,,
  \]
  and hence $w$ is not logarithmic. If $w$ is not torsion-free, then we
  can write $w=w_1w_2$ with every generator in $w_1$ commuting with
  every generator in $w_2$, and $0<|w_1|=k<\infty$. Now,
  \[
  \ell(w^k)=\ell(w_1^kw_2^k)=\ell(w_2^k)<k\cdot \ell(w)\,,
  \]
  and so $w$ is not logarithmic. 
\end{proof}
We ask when the converse of Proposition~\ref{prop:logarithmic}
holds. In 2009, it was shown to hold for Coxeter
elements~\cite{Speyer:09}, and in this paper, we show that it holds
for all CFC elements that lack a certain combinatorial feature called
a ``large band.'' As a corollary, we can conclude that in any group
without ``large odd endpoints,'' a CFC element is logarithmic if and
only if it is torsion free. This class of groups includes all affine Weyl
groups and simply-laced Coxeter groups. Additionally, we give a simple
counterexample when the CFC condition is dropped.

\section{Coxeter and cyclically fully commutative elements}\label{sec:cfc}

A common example of an FC element is a \emph{Coxeter element}, which
is an element for which every generator appears exactly once in each
reduced expression.  The set of Coxeter elements of $W$ is denoted by
$\C(W)$.  As mentioned at the end of the previous section, the converse of Proposition~\ref{prop:logarithmic} holds for
Coxeter elements, and this follows easily from a recent result
in~\cite{Speyer:09} together with the simple fact that Coxeter elements are
trivially cyclically reduced.
\begin{theorem}\label{thm:speyer}
  In any Coxeter group, a Coxeter element is logarithmic if and only
  if it is torsion-free.
\end{theorem}
\begin{proof}
  The forward direction is immediate from
  Proposition~\ref{prop:logarithmic}. For the converse, if $c\in\C(W)$
  is torsion-free, then $c=c_1c_2\cdots c_m$, where each $c_i$ is a
  Coxeter element of an infinite irreducible parabolic subgroup
  $W_{\supp(c_i)}$. Theorem 1 of \cite{Speyer:09} says that in an infinite
  irreducible Coxeter group, Coxeter elements are logarithmic, and it
  follows that for any $k\in\N$,
  \[
  \ell(c^k)=\ell(c_1^k\cdots c_m^k) =\ell(c_1^k)+\cdots+\ell(c_m^k)
  =k\cdot \ell(c_1)+\cdots+k\cdot \ell(c_m)=k\cdot \ell(c)\,,
  \]
  and hence $c$ is logarithmic. 
\end{proof}
The proof of Theorem 1 of~\cite{Speyer:09} is combinatorial, and
relies on a natural bijection between the set $\C(W)$ of Coxeter
elements and the set $\Acyc(\Gamma)$ of acyclic orientations of the
Coxeter graph. Specifically, if $c\in\C(W)$, let $(\Gamma,c)$ denote
the graph where the edge $\{s_i,s_j\}$ is oriented as $(s_i,s_j)$ if
$s_i$ appears before $s_j$ in $c$. (Some authors reverse this
convention, orienting $\{s_i,s_j\}$ as $(s_i,s_j)$ if $s_i$ appears
after $s_j$ in $c$.)  The vertex $s_{x_i}$ is a source (respectively,
sink) of $(\Gamma,c)$ if and only if $s_{x_i}$ is initial
(respectively, terminal) in $c$. Conjugating a Coxeter element
$c=s_{x_1}\cdots s_{x_n}$ by $s_{x_1}$ cyclically shifts the word to
$s_{x_2}\cdots s_{x_n}s_{x_1}$, and on the level of acyclic
orientations, this corresponds to converting the source vertex
$s_{x_1}$ of $(\Gamma,c)$ into a sink, which takes the orientation
$(\Gamma,c)$ to $(\Gamma,s_{x_1}cs_{x_1})$. This generates an
equivalence relation $\tsimkappa$ on $\Acyc(\Gamma)$ and on $\C(W)$,
which has been studied recently in~\cite{Macauley:11c}. Two acyclic
orientations $(\Gamma,c)$ and $(\Gamma,c')$ are $\kappa$-equivalent if
and only if there is a sequence $x_1,\dots,x_k$ such that
$c'=s_{x_k}\cdots s_{x_1}cs_{x_1}\cdots s_{x_k}$ and $s_{x_{i+1}}$ is
a source vertex of $(\Gamma,s_{x_i}\cdots s_{x_1}cs_{x_1}\cdots
s_{x_i})$ for each $i=1,\dots,k-1$.  Thus, two Coxeter elements
$c,c'\in\C(W)$ are $\kappa$-equivalent if they differ by a sequence of
length-preserving conjugations, i.e., if they are conjugate by a word
$w=s_{x_1}\cdots s_{x_k}$ such that
\[
\ell(c)=\ell(s_{x_i}\cdots s_{x_1}cs_{x_1}\cdots s_{x_i})
\]
holds for each $i=1,\dots,k$. Though this is in general a stronger
condition than just conjugacy, the following recent result by
H.~Eriksson and K.~Eriksson shows that they are equivalent for Coxeter
elements, thus establishing the cyclic version of Matsumoto's theorem
for Coxeter elements.
\begin{theorem}[Eriksson--Eriksson \cite{Eriksson:09}]
  \label{thm:conjugacy}
  Let $W$ be a Coxeter group and $c,c'\in\C(W)$. Then $c$ and $c'$ are
  conjugate if and only if $c\tsimkappa c'$. 
\end{theorem}
It is well-known (see~\cite{Stanley:73}) that
$|\!\Acyc(\Gamma)|=T_\Gamma(2,0)$, where $T_\Gamma$ is the Tutte
polynomial~\cite{Tutte:54} of $\Gamma$.  In~\cite{Macauley:08b}, it
was shown that for any undirected graph $\Gamma$, there are exactly
$T_\Gamma(1,0)$ $\kappa$-equivalence classes in
$\Acyc(\Gamma)$. Applying this to Theorem~\ref{thm:conjugacy}, we get
the following result.

\begin{cor}
  In any Coxeter group $W$, the $T_\Gamma(2,0)$ Coxeter elements fall
  into exactly $T_\Gamma(1,0)$ conjugacy classes, where $T_\Gamma$ is
  the Tutte polynomial. \hfill\qed
\end{cor}

The proof of Theorem~\ref{thm:conjugacy} hinges on torsion-free
Coxeter elements being logarithmic, and as mentioned, the proof of
this involves combinatorial properties of the acyclic orientation
construction and source-to-sink equivalence relation. Thus, we are
motivated to extend these properties to a larger class of
elements. Indeed, the acyclic orientation construction above
generalizes to the FC elements. If $w\in\FC(W)$, then $(\Gamma,w)$ is
the graph where the vertices are the disjoint union of letters in any
reduced expression of $w$, and a directed edge is present for each
pair of noncommuting letters, with the orientation denoting which
comes first in $w$. Since $w\in\FC(W)$, the graph $(\Gamma,w)$ is
well-defined. Though the acyclic orientation construction extends from
$\C(W)$ to $\FC(W)$, the source-to-sink operation does not. The
problem arises because a cyclic shift of a reduced expression for an
FC element need not be FC. This motivates the following definition.
\begin{definition}
  An element $w\in W$ is \emph{cyclically fully commutative} (CFC) if
  every cyclic shift of every reduced expression for $w$ is a reduced
  expression for an FC element.
\end{definition}
We denote the set of CFC elements of $W$ by $\CFC(W)$. They are
precisely those whose reduced expressions, when written in a circle,
avoid $\<s,t\>_m$ subwords for $m=m(s,t)\geq 3$, and as such they are the
elements for which the source-to-sink operation extends in a
well-defined manner. However, acyclic directed graphs are not
convenient to capture this generalization -- they are much better
handled as periodic heaps~\cite{Green:07}.

\begin{example}\label{ex:ex1}
  Here are some examples and non-examples of CFC elements.  We will
  return to examples \ref{ex:affine-E6} and \ref{ex:affine-C4} at the
  end of Section~\ref{sec:affineweyl}.
 \begin{enumerate}[label=(\roman*)]
  \item Any Coxeter element is an example of a CFC element,
    because Coxeter elements are FC, and any cyclic
    shift of a Coxeter element is also a Coxeter element.
  \item Consider the Coxeter group of type $A_3$ with generators $s_1,
    s_2, s_3$ labeled so that $s_1$ and $s_3$ commute. The element
    $s_2 s_1 s_3 s_2$ is a reduced expression for an FC element
    $w$. However, $w$ is not cyclically reduced because the above
    expression has a cyclic shift $s_2 s_2 s_1 s_3$ that reduces to
    $s_{1}s_{3}$, and so $w$ is not CFC.
  \item The Coxeter group of type $\widetilde{A}_2$ has
    generators $s_1,s_2,s_3$ with $m(s_i,s_j)=3$ for $i\neq j$. The
    element $s_1s_3s_1s_2$ is cyclically reduced but not FC, because
    $s_1s_3s_1s_2=s_3s_1s_3s_2$. If we increase the bond strength
    $m(s_1,s_3)$ from $3$ to $\infty$, it becomes FC. However, it is
    still not CFC because conjugating it by $s_1$ yields the element
    $s_3s_1s_2s_1=s_3s_2s_1s_2$.
  \item \label{ex:affine-E6} Next, consider the affine Weyl group of
    type $\widetilde{E}_6$ (see Figure~\ref{fig:affine-E6}).
    \begin{figure}[ht]
      \centering
      \begin{tikzpicture}[scale=1.0]
        \draw[fill=black] \foreach \y in {3.75,4.5} {(3,\y) circle (2pt)};
        \draw[fill=black] \foreach \x in {1,2,...,5} {(\x,3) circle (2pt)};
        \draw {(1,3) node[label=below:$s_1$]{}
          (2,3) node[label=below:$s_2$]{}
          (3,3) node[label=below:$s_3$]{}
          (4,3) node[label=below:$s_4$]{}
          (5,3) node[label=below:$s_5$]{}
          (3,3.75) node[label=right:$s_6$]{}
          (3,4.5) node[label=right:$s_0$]{}
          [-] (1,3) -- (5,3)
          [-] (3,3) -- (3,4.5)
        };
      \end{tikzpicture}
      \caption{The Coxeter graph of type $\widetilde{E}_6$.}
      \label{fig:affine-E6}
    \end{figure}
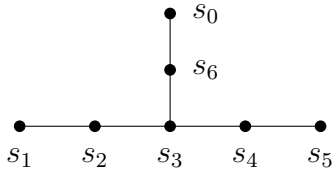
    The element $w=s_1s_3s_2s_4s_3s_5s_4s_6s_0s_3s_2s_6$ is a CFC
    element of $W(\widetilde{E}_6)$, and it turns out that $w$ is logarithmic.
  \item \label{ex:affine-C4} Now, consider the affine Weyl group of
    type $\widetilde{C}_{4}$ (see Figure~\ref{fig:affine-C4}).
    \begin{figure}[ht]
      \centering
      \begin{tikzpicture}[scale=1.0]
        \draw[fill=black] \foreach \x in {1,2,...,5} {(\x,3) circle (2pt)};
        \draw {(1,3) node[label=below:$s_0$]{}
          (2,3) node[label=below:$s_1$]{}
          (3,3) node[label=below:$s_2$]{}
          (4,3) node[label=below:$s_3$]{}
          (5,3) node[label=below:$s_4$]{}
          [-] (1,3) -- (5,3)
        };
        \draw (1.5,3) node[label=above:$4$]{};
        \draw (4.5,3) node[label=above:$4$]{};
      \end{tikzpicture}
      \caption{The Coxeter graph of type $\widetilde{C}_4$.}
      \label{fig:affine-C4}
    \end{figure}
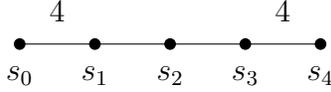
    Let $w_{1}=s_{0}s_{2}s_{4}s_{1}s_{3}$ and
    $w_{2}=s_{0}s_{1}s_{2}s_{3}s_{4}s_{3}s_{2}s_{1}$ be elements in
    $W(\widetilde{C}_{4})$.  It is quickly seen that both elements are
    CFC with full support, and as we shall be able to prove later,
    both $w_{1}$ and $w_{2}$ are logarithmic.
 \end{enumerate}
  
\end{example}

\section{Properties of CFC elements}\label{sec:cfc-properties}

In this section, we will prove a series of results establishing some
basic combinatorial properties of CFC elements. Of particular interest
are CFC elements whose powers are not FC, and we give a complete
characterization of these elements in any Coxeter group. Unless
otherwise stated, $(W,S)$ is assumed to be an arbitrary Coxeter
system. Recall that an expression $\w$ not being FC means that ``$\w$
is not a reduced expression for an FC element,'' i.e., it is either
non-reduced, or it is a reduced expression of a non-FC element. By
Matsumoto's theorem, if $\w\in S^*$ is a reduced expression for a
logarithmic element $w\in W$, then (the group element) $w^k$ is FC if
and only if (the expression) $\w^k$ is FC.

\begin{prop}
  \label{prop:shift}
  If $\w$ is a reduced expression of a non-CFC element of $W$, then
  some cyclic shift of $\w$ is not FC.
  \end{prop}

\begin{proof}
If $\w$ is a reduced expression for a non-CFC element of $w\in W$,
then by definition, a sequence of $i$ cyclic shifts of some reduced
expression $\w'=s_{x_1}\cdots s_{x_m}$ for $w$ produces an expression
$\u=s_{x_{i+1}}\cdots s_{x_m}s_{x_1}\cdots s_{x_i}$ that is either not
reduced, or is a reduced expression for a non-FC element. We may
assume that $w$ itself is FC, otherwise the result is trivial. Thus,
we can obtain $\w'$ from $\w$ via a sequence of $k$ commutations, and
we may take $k$ to be minimal. The result we seek amounts to proving
that $k=0$. By assumption, the expression $\u$ is equivalent via
commutations to one containing either (a) $ss$ or (b)
$\<s,t\>_{m(s,t)}$ as a consecutive subword, where $m(s,t)\geq 3$. For
sake of a contradiction, assume that $k>0$. If the $k$th commutation
(the one that yields $\w'$) does not involve a swap of the letters in
the $i$th and $(i+1)$th positions, then we can simply remove this
commutation from our sequence, because these two letters will be
consecutive in $\u$, and they can be transposed after the cyclic
shifts.  But this contradicts the minimality of $k$.  So, the $k$th
commutation occurs in positions $i$ and $i+1$, sending an expression
$\w''$ to $\w'$, that is,
\begin{equation*}
  \label{eq:kth-commutation}
  \w''=s_{x_1}\cdots s_{x_{i-1}}s_{x_{i+1}}s_{x_{i}}s_{x_{i+2}}\cdots
  s_{x_{m}}\longmapsto s_{x_1}\cdots
  s_{x_{i-1}}s_{x_i}s_{x_{i+1}}s_{x_{i+2}}\cdots s_{x_{m}}=\w'\,.
\end{equation*}
Similarly, if this commutation does not involve one of the generators
in either case (a) or (b), then omitting this commutation before
cyclically shifting still yields an expression that is not FC.  Again,
this contradicts the minimality of $k$, so it must be the case that
the $k$th commutation involves $s$ in case (a) or, without loss of
generality, $s$ in case (b). Moreover, we may assume without loss of
generality that $s_{x_i}=s$, which is in the $(i+1)$th position of
$\w''$ (otherwise, we could have considered $\w^{-1}$, which is
reduced if and only if $\w$ is reduced). Now, apply $i+1$ cyclic
shifts to $\w''$, which yields the element
\[
s_{x_{i+2}}\cdots s_{x_m}s_{x_1}\cdots s_{x_{i-1}}s_{x_{i+1}}s_{x_i}
=s_{x_{i+2}}\cdots s_{x_m}s_{x_1}\cdots s_{x_{i-1}}s_{x_i}s_{x_{i+1}}\in W\,.
\]
Note that this second expression is a single cyclic shift of $\u$. Since $\u$ is commutation equivalent to an expression containing either $ss$ or $\<s,t\>_{m(s,t)}$ as a subword, moving $s_{x_{i+1}}$ (which cannot be $s$ or $t$) from the front of $\u$ to the back does not destroy this property. Thus, we can obtain an expression that is not FC from $\w$ by applying $k-1$ commutations before cyclically shifting, contradicting the minimality of $k$ and completing the proof. 
\end{proof}

\begin{prop}\label{prop:badchain}
  Let $\w$ be an expression that is not FC. Then $\w$ is commutation
  equivalent to an expression of the form $\w_1 \w_2 \w_3$, where
  either $\w_2 = ss$ for some $s \in S$, or $\w_2=\<s,t\>_{m(s,t)}$
  for $m(s,t)\geq 3$.
\end{prop}

\begin{proof}
  This is a restatement of Stembridge's \cite[Proposition
    3.3]{Stembridge:96}. We remark that $\w_{1}$ or $\w_{3}$ could be
  empty. 
\end{proof}

\begin{lemma}
  \label{lem:w^2}
  Let $w\in W$ be logarithmic. If $w^2$ is FC (respectively, CFC),
  then $w^k$ is FC (respectively, CFC) for all $k>2$.
\end{lemma}
\begin{proof}
  Assume without loss of generality that $W$ is irreducible and $w$
  has full support. If $W$ has rank $2$, then $w=(st)^j$ and
  $m(s,t)=\infty$, in which case the result is trivial. Thus, we may
  assume that $W$ has rank $n>2$, and we will prove the
  contrapositive. Let $\w$ be a reduced expression for $w$, and
  suppose that $\w^k$ is not FC (it is reduced because $w$ is
  logarithmic). By Proposition~\ref{prop:badchain}, $\w^k$ is
  commutation equivalent to some $\w_1\w_2\w_3$ where
  $\w_2=\<s,t\>_{m(s,t)}$ with $m(s,t)\geq 3$. Since there is some
  $u\in\supp(w)$ that does not commute with both $s$ and $t$, the
  letters in $\w_2$ can only have come from at most two consecutive
  copies of $\w$ in $\w^k$. Thus, $w^2\not\in\FC(W)$.

  If $w^k\not\in\CFC(W)$, then by Proposition~\ref{prop:shift}, some
  cyclic shift of $\w^k$ is not FC. Since every cyclic
  shift of $\w^k$ is a subword of $\w^{k+1}$, this means that
  $\w^{k+1}$ is not FC. From what we just proved, it follows that
  $w^2\not\in\FC(W)$, and hence $w^2\not\in\CFC(W)$. 
\end{proof}
Observe that the assumption that $w$ is logarithmic is indeed
necessary -- without it, the element $w=s_1s_2$ in $I_2(m)$ for $m\geq
5$ would serve as a counterexample.

\begin{lemma}
  \label{lem:ss}
  Let $W$ be an irreducible Coxeter group of rank $n\geq 2$.  If $\w$
  is a reduced expression for $w\in\CFC(W)$ with full support, then
  $\w^k$ is not commutation equivalent to an expression with $ss$ as a
  subword, for any $s\in S$.
\end{lemma}

\begin{proof}
  For sake of contradiction, suppose that $\w^k$ is commutation
  equivalent to an expression with $ss$ as a subword. Since $\w$ is
  CFC, these two $s$'s must have come from different copies of $\w$ in
  $\w^k$; we may assume consecutive. Thus, we may write
  \[
  \w^2=(\u_1s\w_1)(\u_2s\w_2)\,,\qquad\w=\u_1s\w_1=\u_2s\w_2\,, 
  \]
  where the word $s\w_1\u_2s$ is also commutation equivalent to an
  expression with $ss$ as a subword. There are two cases to
  consider. If $\ell(\u_1)>\ell(\u_2)$, then $s\w_1\u_2s$ is a subword
  of some cyclic shift of $\w$. However, this is impossible because
  $\w$ is CFC. Thus, $\ell(\u_1)\leq\ell(\u_2)$. In this case, some
  cyclic shift of $\w$ is contained in $s\w_1\u_2s$ as a subword, and
  since $\w$ has full support, every generator appears in this
  subword. However, in order for commutations to make the two $s$'s
  consecutive, $s$ must commute with every generator in $\w_1\u_2$,
  which is the required contradiction. 
\end{proof}

There is an analogous result to Lemma~\ref{lem:w^2} when $w$ is not
logarithmic. However, care is needed in distinguishing between the
expression $\w^2$ being FC, and the actual element $w^2$ being FC.

\begin{lemma}
  \label{lem:non-log}
  Let $W$ be an irreducible Coxeter group of rank $n> 2$.  If $\w$ is
  a reduced expression for a non-logarithmic element $w\in\CFC(W)$
  with full support, then $\w^2\not\in\FC(W)$.
\end{lemma}

\begin{proof}
  Pick $k$ so that $\ell(w^k)<k\cdot\ell(w)$. By
  Proposition~\ref{prop:badchain}, $\w^k$ is commutation equivalent to
  some $\w_1\w_2\w_3$ where either $\w_2=ss$, or
  $\w_2=\<s,t\>_{m(s,t)}$ with $m(s,t)\geq 3$. However, the former is
  impossible by Lemma~\ref{lem:ss}. Moreover, there is another
  generator $u\in S$ appearing in $\w$ that does not commute with both
  $s$ and $t$. Therefore, the letters in $\w_2$ can only have come
  from at most two consecutive copies of $\w$ in $\w^k$. Thus,
  $\w^2\not\in\FC(W)$. 
\end{proof}

\begin{prop}
  \label{prop:w^k}
  Let $W$ be an irreducible Coxeter group of rank $n> 2$.  If $\w$ is
  a reduced expression for $w\in\CFC(W)$ with full support and
  $\w^2\in\FC(W)$, then $\w^k\in\CFC(W)$ for all
  $k\in\N$.\footnote{The obvious necessary condition that
    $\w^2\in\FC(W)$ was inadvertently omitted in the journal version.}
\end{prop}
\begin{proof}
  Let $\w$ be a reduced expression for $w\in\CFC(W)$. Since
  $\w^2\in\FC(W)$, Lemma~\ref{lem:non-log} tells us that $w$ is
  logarithmic. Suppose for sake of contradiction, that
  $\w^k\not\in\CFC(W)$ for some $k\geq 2$. By Lemma~\ref{lem:w^2}, we
  know that $\w^2\not\in\CFC(W)$, and by Proposition~\ref{prop:shift},
  some cyclic shift of $\w^2$ is not FC. Every cyclic shift of $\w^2$
  is a subword of $\w^3$, thus $\w^3\not\in\FC(W)$. Applying
  Lemma~\ref{lem:w^2} again gives $\w^2\not\in\FC(W)$, the desired
  contradiction. 
\end{proof}

If $\w$ is a reduced expression of a CFC element and $\w^k$ is FC for
all $k$, then $w$ is clearly logarithmic. Thus, we want to understand
which CFC elements have the property that powers of their reduced
expressions are not FC. Theorem~\ref{thm:bands} gives necessary and
sufficient conditions for this to happen, but first we need more
terminology.  If a vertex $s$ in $\Gamma$ has degree $1$, call it an
\emph{endpoint}. An endpoint vertex (or generator) $s$ has a unique
$t\in S$ for which $m(s,t)\geq 3$, and we call $m(s,t)$ the
\emph{weight} of the endpoint. If this weight is greater than $3$, we
say that the endpoint is \emph{large}. In the remainder of this paper,
we will pay particular attention to ``large odd endpoints,'' that is,
endpoints $s\in S$ for which $m(s,t)$ is odd and at least $5$. (We
will say that $m(s,t)=\infty$ is large but not odd.)  As we shall see,
groups with large odd endpoints have CFC elements with a feature
called a ``large band,'' and these elements have properties not shared
by other CFC elements.

\begin{definition}
  Let $w\in\CFC(W)$ and say that $(W',S')$ is the Coxeter system
  generated by $\supp(w)$. We say that $w$ has an \emph{$st$-band} if
  for some reduced expression $\w$ and distinct generators $s,t\in
  S'$, exactly one of which is an odd endpoint of $(W',S')$, the
  following two conditions hold:
  \begin{enumerate}
  \item some cyclic shift of $\w$ is commutation equivalent to a
    reduced expression containing $\<s,t\>_{m(s,t)-1}$ as a subword;
  \item neither $s$ nor $t$ appears elsewhere in $\w$.
  \end{enumerate}
  We analogously define an \emph{$ts$-band} (i.e., some cyclic shift
  of $\w$ is commutation equivalent to a reduced expression containing
  $\<t,s\>_{m(s,t)-1}$ as a subword).  If we do not care to specify
  whether $s$ or $t$ comes first, then we will simply say that $w$ has
  a \emph{band}. An $st$-band is called \emph{small} if $m(s,t)=3$,
  and \emph{large} otherwise.
\end{definition}

\begin{remark}
  \label{rem:symmetry}
  Note that $w$ has an $st$-band if and only if $w^{-1}$ has a
  $ts$-band. If $w$ has a band, then we may assume, without loss of
  generality, that $w$ has an $st$-band, where $s$ is the odd
  endpoint.
\end{remark}

The following result highlights the importance of bands, and is
essential for establishing our main results on CFC elements.
\begin{theorem}
  \label{thm:bands}
  Let $W$ be an irreducible Coxeter group of rank $n> 2$ and let $\w$
  be a reduced expression for $w\in\CFC(W)$ with full support.  Then
  $\w^k$ is FC for all $k\in\N$ if and only if $w$ has no bands.
\end{theorem}

\begin{proof}
  Suppose that $\w^k$ is not FC for some $k>2$. If $w$ is logarithmic,
  then Lemma~\ref{lem:w^2} tells us that $\w^2$ is not FC. However,
  even if $w$ is not logarithmic, we can still conclude that $\w^2$ is
  not FC, by Lemma~\ref{lem:non-log}. Thus, to prove the theorem, it
  suffices to show that $\w^2$ is not FC if and only if $w$ has a
  band.
 
  First, suppose $\w^2$ is not FC. We will prove that $w$ has a band
  by establishing the following properties:
  \begin{enumerate}[label=(\roman*)]
  \item \label{thm:even has odd} $W$ has an odd endpoint $s$ (say
    $m(s,t)\geq 3$) for which the word $\w^2$ is commutation
    equivalent to an expression of the form
    $\w_1\<s,t\>_{m(s,t)}\w_3$;
  \item \label{thm:even almost non-FC} some cyclic shift of $\w$ is
    commutation equivalent to a reduced expression containing
    $\<s,t\>_{m(s,t)-1}$ or $\<t,s\>_{m(s,t)-1}$ as a subword;
  \item \label{thm:even no extra s or t} neither $s$ nor $t$ appears
    elsewhere in $\w$.
  \end{enumerate}
  Since $\w^2$ is not FC, Proposition~\ref{prop:badchain} implies that
  $\w^2$ is commutation equivalent to an expression of the form $\w_1
  \w_2 \w_3$ in which $\w_2=\<s,t\>_{m(s,t)}$. (Note that $\w_2=ss$ is
  forbidden by Lemma~\ref{lem:ss}.) To prove~\ref{thm:even has odd},
  we will first show that $s$ must be an endpoint, and then show that
  $m(s,t)$ must be odd.

  First, we claim that because $w$ is CFC, two occurrences of $s$ in
  $\w_2$ must correspond to the same letter of $\w$. To see why,
  consider the subword of $\w^2$ from the original position of the
  initial $s$ in $\w_2$ to the original position of the final letter
  (which is either $s$ or $t$). Clearly, the instances of $s$ and $t$
  in this subword must alternate. If no two occurrences of $s$
  correspond to the same letter of $\w$, then this subword is a
  subword of a cyclic shift of $\w$, contradicting the assumption that
  $w$ is CFC, and establishing our claim. In particular, we can write
  $\w^2=(\w'_1s\w'_2)(\w'_1s\w'_2)$, where both instances of $s$ occur
  in $\w_{2}$ and the first instance of $s$ is the initial letter of
  $\w_{2}$. This implies that the letters in $\w'_{2}$ and $\w'_{1}$
  are either other occurrences of $s$ or $t$, or commute with $s$.
  Since $\w=\w'_1s\w'_2$ and has full support and $W$ is irreducible,
  it must be the case that $s$ commutes with every other generator of
  $S$ except $t$, and so $s$ is an endpoint.

  It remains to show that $m(s,t)$ is odd. For sake of a
  contradiction, suppose otherwise, so that $\w_{2}$ ends in $t$.  The
  argument in the previous paragraph using $\w^{-1}$ in place of $\w$
  and $t$ in place of $s$ implies that $t$ must be an endpoint as
  well.  However, we assumed that $W$ is irreducible, and hence $W$
  has rank 2.  This contradicts our assumption that $W$ has rank
  $n\geq 3$, and therefore, $m(s,t)$ is odd.

  To prove \ref{thm:even almost non-FC}, we first prove that the
  instance of $s$ sandwiched between $\w'_{1}$ and $\w'_{2}$ in
  $\w'_{1}s\w'_{2}$ is also the terminal letter of $\w_{2}$.  Towards
  a contradiction, suppose otherwise.  That is, assume that
  $\w^{2}=(\w'_{1}s\u_{1}s\u_{2})(\w'_{1}s\u_{1}s\u_{2})$, where the
  fourth instance of $s$ is the terminal letter of $\w_{2}$.  Then it
  must be the case that every letter between the initial and terminal
  $s$ in $\w_2$ is either $s$, $t$, or a generator that commutes with
  both $s$ and $t$. However, this includes the supports of $\w'_1$,
  $\u_1$ and $\u_2$, and since $\w=\w'_1s\u_1s\u_2$, we conclude that
  every letter in $\w$ is either $s$, $t$, or commutes with $s$ and
  $t$. Again, this contradicts the assumption of $W$ being irreducible
  and of rank $n\geq 3$, so it follows that the two instances of $s$ in
  $(\w'_1s\w'_2)(\w'_1s\w'_2)$ are the initial and terminal letters of
  $\w_{2}$, respectively.  Now, \ref{thm:even almost non-FC} follows
  from the observations that $s\w'_2\w'_1$ is a cyclic shift of $\w$,
  and every $t$ occurring in $\w_{2}$ must occur in
  $\w'_{2}\w'_{1}$. Finally, \ref{thm:even no extra s or t} follows
  from the easy observation that every letter of $\w$ is contained in
  the word $s\w'_2\w'_1s$, which has precisely $m(s,t)$ letters from
  the set $\{s,t\}$. Together, (i), (ii), and (iii) imply that $w$
  has an $st$-band.
  
  We now turn to the converse. Let $w$ be a CFC element with full
  support and a band.  By Remark~\ref{rem:symmetry}, we may assume,
  without loss of generality, that $w$ has an $st$-band, where $s$ is
  the endpoint.  That is, some cyclic shift of $\w$ is commutation
  equivalent to an expression containing $\<s,t\>_{m(s,t)-1}$ as a
  subword. Suppose that $\w=\w_1\w_2$ and the cyclic shift $\w_2\w_1$
  is commutation equivalent to a word $\u=\u_1\<s,t\>_{m(s,t)-1}\u_3$,
  with $\{s,t\}\cap\supp(\u_1\u_3)=\emptyset$. Clearly, $\u^2$ is not
  FC, and so $(\w_2\w_1)^2$ is not FC either. However, $(\w_2\w_1)^2$
  is a subword of $\w^3$, and so $\w^3$ is not FC and hence not
  CFC. By Proposition~\ref{prop:w^k}, $\w^2$ is not FC. 
\end{proof}

\begin{lemma}
  \label{lem:obv}
  Let $W$ be an irreducible Coxeter group with graph $\Gamma$ and let
  $w\in\CFC(W)$.  Let $s, t \in S$ satisfy $m(s, t)\geq 3$, and let
  $\Gamma'$ be the graph obtained from $\Gamma$ by removing the edge
  $\{s,t\}$. Suppose that $\w$ is a reduced expression for $w$ in
  which $t$ occurs exactly once, and that $\Gamma'$ is disconnected.
  Let $\w'$ be the expression obtained from $\w$ by deleting all
  occurrences of generators corresponding to the connected component
  $\Gamma'_s$ of $\Gamma'$ containing $s$.  Then $\w'$ is a reduced
  expression for a CFC element of $W$.
\end{lemma}

\begin{proof}
  Suppose for a contradiction that $\w'$ is not a reduced expression
  for a CFC element. Then either $\w'$ is not a reduced expression, or
  $\w'$ is a reduced expression for a non-CFC element. In the former
  case, $\w'$ is commutation equivalent to an expression $\w''$
  containing either (a) a subword of the form $aa$, or (b) a subword
  of the form $\<a,b\>_{m(a,b)}$ with $m(a, b) \geq 3$. In the latter
  case, Proposition~\ref{prop:shift} implies that $\w'$ can be
  cyclically shifted to yield a non-FC expression. By
  Proposition~\ref{prop:badchain}, this expression is commutation
  equivalent to one with a subword equal to either $aa$ or
  $\<a,b\>_{m(a,b)}$ as in cases (a) and (b) above. Regardless, by
  applying a sequence of commutations or cyclic shifts to $\w'$, we
  can obtain a word $\w''$ containing either $aa$ or
  $\<a,b\>_{m(a,b)}$ (but \emph{not} $\<b,a\>_{m(a,b)}$).

  Since $\w$ does not contain such a subword, it follows in case (a)
  that $a=t$, which is a contradiction because $\w$ contains a unique
  occurrence of $t$.  A similar contradiction arises in case (b),
  except possibly if $b=t$ and $m(a,b)=3$. However, in this case, $a$
  commutes with all generators in $\Gamma'_s$, and so $\w$ would be
  commutation equivalent to an expression with subword of the form
  $aba$.  This contradicts the hypothesis that $w$ is FC, completing
  the proof. \phantom{go right please} 
\end{proof}

Lemma~\ref{lem:obv} has an important corollary -- if a CFC element has
a small band, then the corresponding endpoint can be removed to
create a shorter CFC element.

\begin{cor}
  \label{cor:remove-s}
  Let $\w$ be a reduced expression for $w\in\CFC(W)$.  If $w$ has a 
  small band, then removing the
  corresponding endpoint from $\w$ yields a reduced expression for a
  CFC element $w'$. Moreover, if $w$ has no large bands, then neither
  does $w'$.
\end{cor}

\begin{proof}
  Suppose that $\w$ has a small $st$-band where $s$ is the endpoint.
  By definition, $s$ and $t$ occur uniquely in $\w$. Deleting the edge
  $\{s,t\}$ disconnects the Coxeter graph, and the connected component
  containing $s$ is $\Gamma'_s=\{s\}$. We may now apply
  Lemma~\ref{lem:obv}, to conclude that the word $\w'$ formed from
  deleting the (unique) instance of $s$ is CFC in $W$.

  If $w$ has no large bands, the only way that $w'$ could have a
  large band is if it involved $t$. That is, it would have to be a
  $tu$-band or a $ut$-band for some $u$ where $m(t,u)\geq 5$. However, this
  impossible because $t$ occurs uniquely in $w$, and hence in
  $w'$. 
\end{proof}

It is important to note that Corollary~\ref{cor:remove-s} does not
generalize to large bands. For example, suppose that $s$
is an endpoint with $m(s,t)=3$ and $\w=\w_1st\w_2$ (reduced) is a CFC element
with a small $st$-band. By Corollary~\ref{cor:remove-s}, we can infer
that $\w_1t\w_2$ is CFC. In contrast, suppose that $m(s,t)=5$ and $w$
has a large $st$-band, e.g., $\w=\w_1stst\w_2$ (reduced). Now, it is \emph{not}
necessarily the case that $\w_1t\w_2$, or even $\w_1st\w_2$, is
CFC. Indeed, it may happen that the last letter of $\w_1$ and the
first letter of $\w_2$ are both a common generator $u$ with
$m(t,u)=3$. This peculiar quirk has far-reaching implications -- in
Section~\ref{sec:affineweyl}, will use this deletion property
inductively to give a complete characterization of the logarithmic
CFC elements with no large bands.

\section{Enumeration of CFC elements}\label{sec:cfc-enumeration}

In this section, we will enumerate the CFC elements in all Coxeter
groups. In the groups that contain finitely many, we will also
completely determine the structure of the CFC elements. Once again,
there is a dichotomy between the groups without large odd endpoints
and those with, as the latter class of groups contain CFC elements
with large bands. In~\cite{Stembridge:96}, J.~Stembridge classified
the Coxeter groups that contain finitely many FC elements, calling
them the \emph{FC-finite groups}. In a similar vein, the
\emph{CFC-finite groups} can be defined as the Coxeter groups that
contain only finitely many CFC elements. Our next result shows that a
group is CFC-finite if and only if it is FC-finite. The Coxeter graphs
of these (irreducible) groups are shown in Figure~\ref{fig:cfcfg}, and
they comprise seven infinite families. (The vertex labeled $s_0$ is
called the \emph{branch vertex}, and will be defined later.)

\begin{theorem}\label{thm:cfcfinite}
  The irreducible CFC-finite Coxeter groups are $A_n$ ($n\geq 1$),
  $B_n$ ($n\geq 2$), $D_n$ ($n\geq 4$), $E_n$ ($n\geq 6$), $F_n$
  ($n\geq 4$), $H_n$ ($n\geq 3$), and $I_2(m)$ ($5\leq
  m<\infty$). Thus, a Coxeter group is CFC-finite if and only if it is
  FC-finite.
\end{theorem}

\begin{proof}
  The ``if'' direction is immediate since $\CFC(W)\subseteq\FC(W)$, so
  it suffices to show that every CFC-finite group is
  FC-finite. Stembridge classified the FC-finite groups
  in~\cite{Stembridge:96} by classifying their Coxeter graphs. In
  particular, he gave a list of ten forbidden properties that an
  FC-finite group cannot have. The list of FC-finite groups is
  precisely those that avoid all ten of these obstructions. The first
  five conditions are easy to state, and are listed below.
  \begin{enumerate}
  \item $\Gamma$ cannot contain a cycle.
  \item $\Gamma$ cannot contain an edge of weight $m(s,t)=\infty$.
  \item $\Gamma$ cannot contain more than one edge of weight greater than 3.
  \item $\Gamma$ cannot have a vertex of degree greater than 3, or more than
    one vertex of degree 3.
  \item $\Gamma$ cannot have both a vertex of degree 3 and an edge of
    weight greater than 3.
  \end{enumerate}
  The remaining five conditions all require the definition of a heap,
  and in the interest of space, will not be stated here. For each of
  the ten conditions, including the above five, Stembridge shows that
  if it fails, one can produce a word $\w\in W$ such that $\w^k$ is FC
  for all $k\in\N$. This, together with Proposition~\ref{prop:w^k},
  implies that if $W$ is CFC-finite, then it is FC-finite, and the
  result follows immediately.  
\end{proof}

\begin{figure}[ht]
\centering
\begin{tabular}{ll}
\begin{tikzpicture}[scale=1.0]
\draw[fill=black] \foreach \x in {1,2} {(\x,0) circle (2pt)};
\draw {(.25,0) node{$I_{2}(m)$}
(1.5,0) node[label=above:$m$]{}
[-] (1,0) -- (2,0)
(2,0) node[label=below:$s_{0}$]{}};
\end{tikzpicture}
& \\
\begin{tikzpicture}[scale=1.0]
\draw[fill=black] \foreach \x in {1,2,...,6} {(\x,10) circle (2pt)};
\draw {(.5,10) node{$A_{n}$}
(4.5,10) node{$\cdots$}
[-] (1,10) -- (4,10)
[-] (5,10) -- (6,10)
(1,10) node[label=below:$s_{0}$]{}}; 
\end{tikzpicture}
&
\quad  \quad \begin{tikzpicture}[scale=1.0]
\draw[fill=black] \foreach \x in {1,2,...,6} {(\x,4.5) circle (2pt)};
\draw[fill=black] (3,5.5) circle (2pt);
\draw {(.5,4.5) node{$E_{n}$}
(4.5,4.5) node{$\cdots$}
[-] (1,4.5) -- (4,4.5)
[-] (5,4.5) -- (6,4.5)
[-] (3,4.5) -- (3,5.5)
(3,4.5) node[label=below:$s_{0}$]{}};
\end{tikzpicture}\\
\\
\begin{tikzpicture}[scale=1.0]
\draw [fill=black] \foreach \x in {1,2,...,6} {(\x,8.5) circle (2pt)};
\draw {(.5,8.5) node{$B_{n}$}
(1.5,8.5) node[label=above:$4$]{}
(4.5,8.5) node{$\cdots$}
[-] (1,8.5) -- (4,8.5)
[-] (5,8.5) -- (6,8.5)
(2,8.5) node[label=below:$s_{0}$]{}}; 
\end{tikzpicture}
&
\quad  \quad \begin{tikzpicture}[scale=1.0]
\draw[fill=black] \foreach \x in {1,2,...,6} {(\x,3) circle (2pt)};
\draw {(.5,3) node{$F_{n}$}
(2.5,3) node[label=above:$4$]{}
(4.5,3) node{$\cdots$}
[-] (1,3) -- (4,3)
[-] (5,3) -- (6,3)
(3,3) node[label=below:$s_{0}$]{}};
\end{tikzpicture}\\
\\
\begin{tikzpicture}[scale=1.0]
\draw[fill=black] \foreach \x in {1,2,...,6} {(\x,6.5) circle (2pt)};
\draw[fill=black] (2,7.5) circle (2pt);
\draw {(.5,6.5) node{$D_{n}$}
(4.5,6.5) node{$\cdots$}
[-] (1,6.5) -- (4,6.5)
[-] (5,6.5) -- (6,6.5)
[-] (2,6.5) -- (2,7.5)
(2,6.5) node[label=below:$s_{0}$]{}};
\end{tikzpicture}
&
\quad  \quad \begin{tikzpicture}[scale=1.0]
\draw[fill=black] \foreach \x in {1,2,...,6} {(\x,1.5) circle (2pt)};
\draw {(.5,1.5) node{$H_{n}$}
(1.5,1.5) node[label=above:$5$]{}
(4.5,1.5) node{$\cdots$}
[-] (1,1.5) -- (4,1.5)
[-] (5,1.5) -- (6,1.5)
(2,1.5) node[label=below:$s_{0}$]{}}; 
\end{tikzpicture}
\end{tabular}
\caption{Connected Coxeter graphs corresponding to CFC-finite groups.}
\label{fig:cfcfg}
\end{figure}
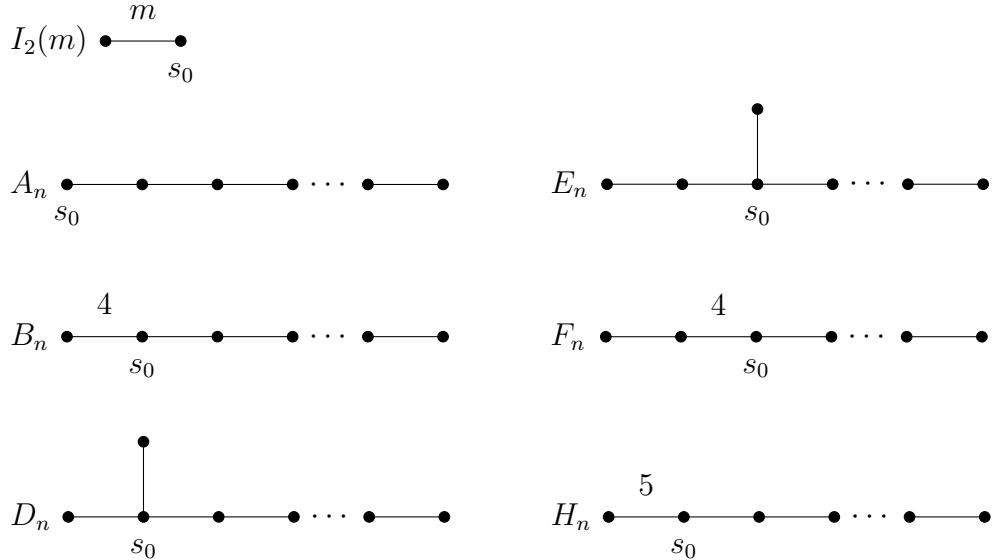

We now turn our attention to enumerating the CFC elements in the
CFC-finite groups. The following lemma is well-known, but we are not
aware of a suitable reference, so we provide a proof here.

\begin{lemma}
  \label{lemma:uniqua}
  Let $W$ be a Coxeter group of type $A_n$ and let $s$ be an endpoint
  generator of $A_n$. If $\w$ is a reduced expression for $w \in
  \FC(W)$, then $s$ occurs at most once in $\w$.
\end{lemma}

\begin{proof}
  We may assume that $s$ occurs in $\w$, and by symmetry, we may assume
  that $s = s_n$.
  
  In type $A_n$, a well-known reduced expression for the longest element
  $w_0$ is
  \[
  s_1 (s_2 s_1) (s_3 s_2 s_1) \cdots (s_n s_{n-1} \cdots s_1).
  \]
  Every element of $w$ satisfies $w \leq w_0$ with respect to the
  Bruhat order, which means that any such $w$ may be written as a
  subexpression of the given expression.  In particular, any element
  $w$ has a reduced expression containing at most one occurrence of
  $s_n$.  This applies to the case where $w \in \FC(W)$, in which case
  one (and hence all) reduced expressions for $w$ contain at most one
  occurrence of $s_n$. 
\end{proof}

\begin{lemma}
  \label{lemma:cfch}
  Let $W$ be a Coxeter group of type $H_n$.  Label the elements of $S$ as
  $s_1, s_2, \ldots, s_n$ in the obvious way such that $m(s_1, s_2) = 5$.  Let 
  $\w$ be a reduced expression for an element $w \in \CFC(H_n)$ having full
  support.  Then the following all hold:
 \begin{enumerate}[label=(\roman*)]
 \item[(i)] $\w$ contains precisely one occurrence of each generator 
   $s_i$ for $i \geq 3$;
 \item[(ii)] $\w$ contains precisely $j$ occurrences of each generator $s_1$
   and $s_2$, where $j \in \{1, 2\}$;
 \item[(iii)] if $\w$ is not a Coxeter element, then it has a large band.
\end{enumerate}
\end{lemma}

\begin{proof}
  We prove (i) and (ii) by induction on $n$. For both, the base case
  is $n=2$, which follows by a direct check of $W(I_2(5))$. We will
  prove (i) first, and will assume that $n>2$. From
  Theorem~\ref{thm:cfcfinite}, we know that $W$ has finitely many CFC
  elements.  It follows that for some $k\in \N$ (actually, $k=2$
  works, but this is unimportant), $\w^k$ is not FC, and so by
  Theorem~\ref{thm:bands}, $w$ has a band. Thus, $w$ has a reduced
  expression $\w$ that can be cyclically shifted to a word that is
  commutation equivalent to an expression $\u$ containing either
  $s_1s_2s_1s_2$ or $s_{n-1}s_n$ as a subword (by
  Remark~\ref{rem:symmetry}, we can disregard the other two cases:
  $s_2s_1s_2s_1$ and $s_ns_{n-1}$).

  First, suppose $\w$ has an $s_1s_2$-band, so
  $\u=\u_1s_1s_2s_1s_2\u_2$, and
  $\{s_1,s_2\}\cap\supp(\u_1\u_2)=\emptyset$. Since $\w$ is CFC,
  $\u_2\u_1$ is FC. This element sits inside a type $A_{n-2}$
  parabolic subgroup of $W$ of which $s_3$ is an endpoint. By
  Lemma~\ref{lemma:uniqua}, $s_3$ occurs uniquely in $\u_2\u_1$. Now
  consider the word $\u_1\u_2$. By Lemma~\ref{lem:obv} applied to $\w$
  and the pair of generators $\{s_2, s_3\}$, we see that $\u_1\u_2$ is
  CFC, and we already know that it contains a unique instance of
  $s_3$. By repeated applications of Corollary~\ref{cor:remove-s} and
  the fact that type $A$ is finite, we deduce that $\u_1\u_2$ contains
  precisely one occurrence of each generator in the set $\{s_3, s_4,
  \ldots, s_n\}$, and this proves (i).

  For (ii), assume again that $n>2$ and suppose that $w$ has no large
  band, meaning it must have an $s_{n-1}s_n$-band. We may use
  Corollary~\ref{cor:remove-s} to delete the (unique) occurrence of
  $s_n$ from $\w$ to obtain a CFC element of $W(H_{n-1})$ also having
  full support and no large band. The result now follows by
  induction. 

  For (iii), assume that $\w$ is CFC but not a Coxeter element, and
  $n>2$. By (i) and (ii), $s_1$ and $s_2$ must occur in $\w$ twice
  each, and $s_3$ can only occur once. Clearly, $\w$ is a cyclic shift
  of a CFC element beginning with $s_3$, and since this is the only
  occurence of $s_3$ (the only generator that does not commute with
  both $s_1$ and $s_2$), this element is commutation equivalent to one
  containing either $s_1s_2s_1s_2$ or $s_2s_1s_2s_1$ as a
  subword. Therefore, $w$ has a large band. 
\end{proof}

Suppose $\Gamma$ is the Coxeter graph for an irreducible CFC-finite Coxeter
group. Define $\Gamma_0$ to be the type $A$ subgraph of $\Gamma$
consisting of (a) the generator $s_0$ as labeled in Figure~\ref{fig:cfcfg} and (b) everything to the right of it.  
We call $\Gamma_0$ the \emph{branch} of $\Gamma$ and refer to the distinguished vertex $s_0$ as the \emph{branch vertex}.

The FC elements in the FC-finite groups can be quite complicated to
describe (see~\cite{Stembridge:96,Stembridge:98}). In contrast, the
CFC elements have a very restricted form. The following result shows
that except in types $H_n$ and $I_2(m)$, they are just the Coxeter elements.
\begin{prop}
  \label{prop:CFCelementsfull}
  Let $W$ be an irreducible CFC-finite group.  Suppose that
  $w\in\CFC(W)$ has full support, and that some generator $s\in S$
  appears in $w$ more than once.  Then one of the following situations
  occurs.
  \begin{enumerate}[label=(\roman*)]
    \item $W = I_2(m)$ and $w = stst\cdots st$ has even length and 
    satisfies $0 \leq \ell(w) < m$, or
    \item $W = H_n$ for $n > 2$, and $w$ has a large band.
  \end{enumerate}
\end{prop}
\begin{proof}
The proof is by induction on $|S|=n$, the case with $n = 1$ being
trivial. If $n = 2$, then $W = I_2(m)$.  In this case, it is easily
checked that the CFC elements are those of the form $w = stst\cdots
st$, where $s$ and $t$ are distinct generators, $\ell(w)$ is even, and
$0 \leq \ell(w) < m=m(s, t)$.

Suppose now that $n>2$. The case when $W=H_n$ follows from
Lemma~\ref{lemma:cfch}. For all other cases,
Theorem~\ref{thm:cfcfinite} tells us that $W$ has no large odd
endpoints. Let $\w$ be a reduced expression for $w$. Since $W$ is
CFC-finite, there exists $k \in \N$ such that $\w^k$ is not FC. In
this case, it follows by induction on rank and
Corollary~\ref{cor:remove-s} that $w$ is a Coxeter element, which is a
contradiction. 
\end{proof}

\begin{remark}
If $w\in\CFC(W)$ with full support such that $W\neq I_{2}(m),H_{n}$, then $w$ must be a Coxeter element.
\end{remark}

Finally, we can drop the restriction that $w$ should have full support.

\begin{cor}
  \label{cor:CFCelements}
  Let $W$ be an irreducible CFC-finite group. 
  Suppose that $w\in\CFC(W)$, and that some
  generator $s\in S$ appears in $w$ more than once.  Then there exists a unique
  generator $t\in S$ with $m(s, t) \geq 5$.  Furthermore, the generators
  $s$ and $t$ occur $j$ times each, in alternating order (but not
  necessarily consecutively), where $2j < m(s, t)$.
\end{cor}

\begin{proof}
  This follows from Proposition~\ref{prop:CFCelementsfull} by
  considering the parabolic subgroup corresponding to $\supp(w)$, and
  considering each connected component of the resulting Coxeter graph.
  \phantom{go right please} 
\end{proof}

Corollary~\ref{cor:CFCelements} allows us to enumerate the CFC elements of
the CFC-finite groups. Let $W_n$ denote a rank-$n$ irreducible CFC-finite group
of a fixed type, where $n\geq 3$, and let $W_{n-1}$ be the parabolic
subgroup generated by all generators except the rightmost generator
of the branch of $W_n$.
\begin{cor}
  \label{cor:recurrence}
  Let $n\geq 4$. If $\alpha_n=|\CFC(W_n)|$, then $\alpha_n$ satisfies the
  recurrence
  \begin{equation}
    \label{eq:recurrence}
    \alpha_n=3\alpha_{n-1}-\alpha_{n-2}\;.
  \end{equation}
\end{cor}

\begin{proof}
  The base cases can be easily checked by hand for each type. Every
  CFC element in $W_{n-1}$ is also CFC in $W_{n}$, and there are
  $\alpha_{n-1}$ of these. Let $s$ be the rightmost generator of the
  branch of $W_n$, and consider the CFC elements that contain $s$.  By
  Proposition~\ref{prop:CFCelementsfull}, $s$ and the unique generator
  $t$ such that $m(s,t)\geq 3$ occur at most once each.  This implies
  that every element can be written as $sw$ or $ws$ (both reduced),
  thus we need to compute the cardinality of
  \[
  \{sw\mid w\in\CFC(W_{n-1})\}\cup\{ws\mid w\in\CFC(W_{n-1})\}\;.
  \]
  Each of these two sets has size $\alpha_{n-1}$, and $sw=ws$ if and
  only if $s_{n-1}\not\in\supp(w)$. Thus, their intersection has size
  $|\CFC(W_{n-2})|=\alpha_{n-2}$, and their union has size
  $2\alpha_{n-1}-\alpha_{n-2}$. In summary, there are
  $2\alpha_{n-1}-\alpha_{n-2}$ CFC elements that contain $s$, and
  $\alpha_{n-1}$ CFC elements that do not, so
  $\alpha_n=3\alpha_{n-1}-\alpha_{n-2}$.  
\end{proof}

\begin{remark}
If one restricts attention to CFC elements with full support, then there
is a version of Corollary~\ref{cor:recurrence} for which the recurrence
relation is $\alpha_n = 2\alpha_{n-1}$ for sufficiently large $n$.
\end{remark}

By Corollary~\ref{cor:recurrence}, to enumerate the CFC elements in
$W_n$ for each type, we just need to count them in the smallest groups
of that family. We will denote the number of CFC elements in the
rank-$n$ Coxeter group of a given type by the corresponding lowercase
letter, e.g., $b_n=|\CFC(B_n)|$. Table~\ref{tbl:enumeration} contains
a summary of the results of each (non-dihedral) type, up to
$n=9$. It also lists the number of FC elements in each type, which was
obtained in~\cite{Stembridge:98}. It is interesting to note that the
enumeration of the FC elements is quite involved, and uses a variety
of formulas, recurrences, and generating functions. In contrast, the
CFC elements in these groups can all be described by the same simple
recurrence (except in type $I_2(m)$, which is even easier).

\subsection{Type $A$}

The elements of $A_1=\{1,s\}$ have orders 1 and 2, respectively, and
the set of CFC elements in $A_2 = I_2(3)$ is $\{1,s,t,st,ts\}$.  It
follows that $a_1=2$ and $a_2=5$. The odd-index Fibonacci numbers
satisfy the recurrence in~\eqref{eq:recurrence} as well as the initial
seeds (see~\cite[A048575]{Sloane}). Therefore, $a_n=\Fib_{2n-1}$,
where $\Fib_k$ denotes the $k$th Fibonacci number.
By Corollary~\ref{cor:CFCelements}, the CFC elements in $A_n$ are
precisely those that have no repeat generators. In the language
of~\cite{Tenner:07}, these are the Boolean permutations, and are
characterized by avoiding the patterns $321$ and $3412$. (A
permutation $\pi$ avoids $3412$ if there is no set $\{i,j,k,l\}$ with
$i<j<k<\ell$ and $\pi(k)<\pi(\ell)<\pi(i)<\pi(j)$.) The following
result is immediate.
\begin{cor}
  An element $w\in A_n$ is CFC if and only if $w$ is $321$- and
  $3412$-avoiding.
\end{cor}
It is worth noting that $\Fib_{2n-1}$ also counts the $1324$-avoiding
\emph{circular} permutations on $[n+1]$
(see~\cite{Callan:02}). Roughly speaking, a circular permutation is a
circular arrangement of $\{1,\dots,n\}$ up to cyclic shift.  Though
$\Fib_{2n-1}$ counts the circular permutations that avoid $1324$,
these are set-wise not the same as the CFC elements in
$W(A_n)=\sym_{n+1}$. As a simple example, the permutation
$(2,3)=s_{2}\in W(A_3)$ does not avoid $1324$ since it equals $[1324]$
in 1-line notation, but is clearly CFC. Also, the element
$s_2s_3s_1s_2s_4s_3\in W(A_4)$ (or $(1,3,5,2,4)$ in cycle notation)
has no (circular) occurrence of $1324$, but is not CFC.

\begin{table}
  \centering\small
  \begin{tabular}{|lc|crrrrrrrr|} \hline
    & Type  & $n=1$ & 2 & 3 &  4 &   5  &   6 &   7  &    8 & 9  \\ \hline
    $\#\FC$ & $A$  & 2 & 5 & 14 & 42 & 132 & 429 & 1430 & 4862 & 16796  \\
    $\#\FC$ & $B$ & 2 & 7 & 24 & 83 & 293 & 1055 & 3860 & 14299 & 53481 \\
    $\#\FC$ & $F$ & 2 & 5 & 24 & 106 & 464 & 2003 & 8560 & 36333 & 153584 \\
    $\#\CFC$ & $A,B,F$ & 2 & 5 & 13 & 34 &  89 & 233 &  610 & 1597 & 4181 \\ 
    \hline
    $\#\FC$ & $D$  & 2 & 4 & 14 & 48 & 167 & 593 & 2144 & 7864 & 29171  \\
    $\#\CFC$ & $D$ & 2 & 4 & 13 & 35 & 92  & 241 &  631 & 1652  & 4325  \\ 
    \hline
    $\#\FC$ & $E$ &&& 10 & 42 & 167 & 662 & 2670 & 10846 & 44199  \\
    $\#\CFC$ & $E$ &&& 10 & 34 & 92 & 242 & 634 & 1660   &  4346  \\ \hline
    $\#\FC$ & $H$  & 2 & 9 & 44 & 195 & 804 & 3185 & 12368 & 47607 & 182720  \\
    $\#\CFC$ & $H$ & 2 & 7 & 21 &  56 & 147 &  385 & 1008 & 2639 & 6909  \\ 
    \hline
\end{tabular} \vspace*{2mm}
\caption{The number of FC and CFC elements in the CFC-finite groups,
by their rank $n$.
\label{tbl:enumeration}}
\end{table}

\subsection{Type $B$}

The two elements of $B_1$ have orders 1 and 2. In $B_2=I_2(4)$, the
elements $sts$ and $tst$ are not cyclically reduced. All remaining
elements other than the longest element are CFC, so we have $b_1=2$
and $b_2=5$.

\subsection{Type $D$}

The group $D_1$ is isomorphic to $A_1$, $D_2$ has two commuting Coxeter
generators, and $D_3$ is
isomorphic to $A_3$. Therefore, $d_1 = 2$, $d_2=4$ and $d_3=13$.

\subsection{Type $E$}

The groups $E_4$ and $E_5$ are isomorphic to $A_4$ and $D_5$,
respectively, and so $e_4=34$ and $e_5=92$. We note that if we define
$E_3$ by removing the branch vertex from the Coxeter graph of $E_4$,
leaving an edge and singleton vertex, then is is readily checked that
$e_3=10$, and so $e_5=3e_4-e_3$. 

\subsection{Type $F$}

The groups $F_2$ and $F_3$ are isomorphic to $A_2$ and $B_3$,
respectively, and so $f_2=5$ and $f_3=13$. As in Type $E$, if we
define $F_1$ as having a singleton Coxeter graph, then $f_1=2$, and
$f_3=3f_2-f_1$. Thus, these are also counted by the odd-indexed
Fibonacci numbers with a ``shifted'' seed, yielding $f_n=\Fib_{2n+1}$.

\subsection{Type $H$}

The group $H_1$ has order $2$, and in $H_2=I_2(5)$, the elements $sts$
and $tst$ are not cyclically reduced. All other elements except the
longest element are CFC, so $h_1=2$ and $h_2=7$.

\section{The Root Automaton}\label{sec:rootaut}

In order to prove our main result, Theorem~\ref{thm:largebands},
we will induct on the size of the generating set $S$.  A key part in
the inductive step is Lemma \ref{lemma:sinsert}, which shows that in
certain circumstances, one can insert occurrences of a new generator
into an existing reduced expression in such a way as to make a new
reduced expression.  To do this, we use the \emph{root automaton}.  This
technique is described in \cite[Chapters 4.6--4.9]{Bjorner:05}, and
has recently been used to tackle problems similar to ours by
H. Eriksson and K. Eriksson \cite{Eriksson:09}. We formalize it
differently, though, in a way that is useful for our purposes, and
should be of general interest in its own right.

For a Coxeter system $(W,S)$ on $n$ generators, let $V$ be an
$n$-dimensional real vector space with basis
$\{\vec{\alpha}_1,\dots,\vec{\alpha}_n\}$, and equip $V$ with a
symmetric bilinear form $B$ such that
$B(\vec{\alpha}_i,\vec{\alpha}_j)=-\cos(\pi/m_{i,j})$. The action of
$W$ on $V$ by
$s_i\colon\vec{v}\mapsto\vec{v}-2B(\vec{v},\vec{\alpha}_i)\vec{\alpha}_i$
is faithful and preserves $B$, and the elements of the set
$\Phi=\{w\vec{\alpha}_i\mid w\in W\}$ are called \emph{roots}. The map
\[
W\longrightarrow\GL(V)\,,\qquad
s_i\longmapsto\big(\vec{v}\stackrel{F_i}{\mapsto}
\vec{v}-2B(\vec{v},\vec{\alpha}_i)\vec{\alpha}_i\big)
\]
is called the standard geometric representation of $W$. Henceforth, we
will let $\vec{\alpha}_i=\e_i\in\R^n$, the standard unit basis vector,
hereby identifying roots of $W$ with vectors in $\R^n$. Partially
ordering the roots by $\leq$ componentwise yields the \emph{root
  poset} of $W$. For any $\z=(z_1,\dots,z_n)\in\R^n$, the action of
$W$ on $\Phi$ is given by
\begin{equation}
\label{eq:rootautomaton}
\z\stackrel{s_i}{\longmapsto} \z+\sum_{j=1}^n 2\cos(\pi/m_{i,j})
z_j\e_i\;.
\end{equation}
In summary, the action of $s_i$ flips the sign of the $i^{\rm th}$
entry and adds each neighboring entry $z_j$ weighted by
$2\cos(\pi/m_{i,j})$. It is convenient to view this as the image of
$s_i$ under the standard geometric representation $W\to\GL(\R^n)$,
which is a linear map $F_i\colon \R^n\to\R^n$ defined by
\begin{equation}
\label{eq:F_i}
F_i\colon (z_1,\dots,z_n)\longmapsto
(z_1,\dots,z_{i-1},z_i+\sum_{j=1}^n 2\cos(\pi/m_{i,j})
z_j,z_{i+1},\dots,z_n)\;.
\end{equation}
Similarly, for any $\w=s_{x_1}\cdots s_{x_k}\in S^*$, let
$F_{\w}=F_{s_{x_k}}\circ\cdots\circ F_{s_{x_1}}$. It is well-known that for
every root, all non-zero entries have the same sign, thus the root
poset consists of positive roots $\Phi^+$ and negative roots $\Phi^-$,
with $\Phi=\Phi^+\cup\Phi^-$. In 1993, Brink and Howlett proved that
Coxeter groups are automatic~\cite{Brink:93}, guaranteeing the
existence of an automaton for detecting reduced expressions (see
also~\cite{Bjorner:05,ErikssonH:94}). This \emph{root automaton} has
vertex set $\Phi$ and edge set $\{(\z,s_i\z)\mid\z\in\Phi,\,s_i\in
S\}$. For convenience, label each edge $(\z,s_i\z)$ with the
corresponding generator $s_i$. It is clear that upon disregarding
loops and edge orientations (all edges are bidirectional anyways), we
are left with the Hasse diagram of the root poset. We represent a word
$\w=s_{x_1}s_{x_2}\cdots s_{x_m}$ in the root automaton by starting at
the unit vector $\e_{x_1}\in\Phi^+$ and traversing the edges labeled
$s_{x_2},s_{x_3},\dots,s_{x_m}$ in sequence. Denote the root reached
in the root poset upon performing these steps by $\r(\w)$. The
sequence
\[
\e_{x_1}=\r(s_{x_1}),\r(s_{x_1}s_{x_2}),\dots,\r(s_{x_1}s_{x_2}\cdots
s_{x_m})=\r(\w)\, 
\]
is called the \emph{root sequence} of $\w$.  If
$\r(s_{x_1}s_{x_2}\cdots s_{x_i})$ is the first negative root in the
root sequence for $\w$, then a shorter expression for $\w$ can be
obtained by removing $s_{x_1}$ and $s_{x_i}$. By the exchange property
of Coxeter groups~(see \cite{Bjorner:05}), every non-reduced word
$\w\in S^*$ can be made into a reduced expression by iteratively
removing pairs of letters in this manner. Clearly, the word
$\w=s_{x_1}\cdots s_{x_m}\in S^*$ is reduced if and only if
$\r(s_{x_i}s_{x_{i+1}}\cdots s_{x_j})\in\Phi^+$ for all $i<j$.

We say that a Coxeter system $(W',S)$ \emph{dominates} $(W,S)$ if each
bond strength in $(W',S)$ is at least as large as the corresponding
bond strength in $(W,S)$.

\begin{lemma}
  \label{lem:monotone}
  Suppose $(W', S)$ dominates $(W,S)$ and let $\w$ be a reduced expression for 
  $w\in W$.  Then $\w$ is reduced in $W'$, as well.
\end{lemma}

\begin{proof}
  This is a consequence of Matsumoto's Theorem. 
\end{proof}

The following lemma is reminiscent of \cite[Proposition 3.3]{Eriksson:09}.

\begin{lemma}\label{lemma:sinsert}
  Suppose that $W'$ is obtained from $W$ by adding a new generator $s$
  to $S$, setting $m(s,t)\geq 3$ for some $t\in S$, and $m(s,s')=2$
  for all $s'\neq t$. Let $\w_i$ be a reduced expression for $w_i\in
  W$, and suppose that $\w_1\w_2\cdots\w_n$ is reduced, and that each
  of $\w_2,\dots,\w_{n-1}$ contains at least one occurrence of $t$.
  Then $\w_1s\w_2s\w_3\cdots s\w_n$ is a reduced expression for an
  element of $W'$.
\end{lemma}

\begin{proof}
  It suffices to show that $\r(\w_1s\w_2s\w_3\cdots s\w_n)$
  is a positive root, and we will induct on $n$. Moreover, by
  Lemma~\ref{lem:monotone}, we only need to prove it for the case when
  $m(s,t)=3$.

  The base case is when $n=3$, because this guarantees at least one
  instance of $t$ in $\w_1s\w_2s\w_3$. First, observe that $s\w_2s$
  is reduced, because $s\not\in D_R(s\w_2)$. Also, note that
  $\r(\w_1s)=\r(\w_1)+c_1\e_s=\r(\w_1)+c_1\r(s)$, for some
  non-negative constant $c_1$. By linearity,
  \begin{align*}
    \r(\w_1s\w_2s\w_3)&=F_{\w_3}\circ F_s\circ F_{\w_2}[\r(\w_1s)]\\
    &=F_{\w_3}\circ F_s\circ F_{\w_2}[\r(\w_1)+c_1\r(s)]\\ 
    &=\r(\w_1\w_2s\w_3)+c_1\r(s\w_2s\w_3)\,.
  \end{align*}
  It suffices to show that both of these roots are positive, or
  equivalently, that the corresponding words are reduced. First off,
  $\w_1\w_2s\w_3$ is clearly reduced in the Coxeter group formed by
  setting $m(s,t)=2$, and so it is reduced in $W'$ by
  Lemma~\ref{lem:monotone}. We now turn our attention to
  $\r(s\w_2s\w_3)$.  Suppose that $\w_2=\u_0t\u_1t\u_2\cdots t\u_k$,
  with $t\not\in\supp(\u_i)$ for each $i$ (by assumption, $i\geq
  1$). Since $s$ is disjoint from all vertices in each $\u_i$, we have
  $\r(s\u_i)=\r(s)$. Thus, we may omit $\u_0$ from $\w_2$ when
  computing $\r(s\w_2s\w_3)$. Since $m(s,t)=3$, we have 
  $\r(st)=\r(t)+\r(s)$, and so
  \begin{align*}
    \r(s\w_2s)=\r(st\u_1t\u_2\cdots t\u_ks)&=F_{\u_1t\u_2\cdots t\u_k s}[\r(t)+\r(s)] \\
    &=\r(t\u_1t\u_2\cdots t\u_ks)+\r(s\u_1t\u_2\cdots t\u_ks) \\
    &=\r(t\u_1t\u_2\cdots t\u_ks)+\r(st\u_2\cdots t\u_ks)\,.
  \end{align*}
  Applying this same technique to $\r(st\u_2\cdots t\u_ks)$ yields
  \[
  \r(st\u_2\cdots t\u_ks)=F_{\u_2t\u_3\cdots t\u_k s}[\r(t)+\r(s)]
  =\r(t\u_2t\u_3\cdots t\u_ks)+\r(st\u_3\cdots t\u_ks)\,.
  \]
  We can continue this process and successively pick off roots of the form
  $\r(t\u_i\cdots t\u_ks)$ for $i=1,2,\dots$. At the last step,
  we get
  \[
  \r(st\u_ks)=F_{\u_ks}[\r(t)+\r(s)]=\r(t\u_ks)-\r(s)=[\r(t\u_k)+\r(s)]-\r(s)
  =\r(t\u_k)\,.
  \]
  Putting this together, we have
  \begin{align*}
    \r(s\w_2s)&=\r(s\u_0t\u_1\cdots t\u_ks)\\
    &=\r(st\u_1\cdots t\u_ks) \\
    &=[\r(t\u_1\cdots t\u_ks)+\cdots+\r(t\u_{k-1}t\u_ks)+\r(t\u_ks)]
    - \r(s) \\
    &=[\r(t\u_1\cdots t\u_ks)+\cdots+\r(t\u_{k-1}t\u_ks)]+\r(t\u_k)\,.
  \end{align*}
  Finally, we get $\r(s\w_2s\w_3)$ from this by applying the map
  $F_{\w_3}$ to each term, yielding
  \begin{equation}\label{eq:basecase}
  \r(s\w_2s\w_3)
  =[\r(t\u_1\cdots t\u_ks\w_3)+\cdots+\r(t\u_{k-1}t\u_ks\w_3)]+\r(t\u_k\w_3)\,.
  \end{equation}
  Each of the roots on the right-hand side of~\eqref{eq:basecase} are
  roots of expressions that are subwords of $\w_2s\w_3$ or $\w_2\w_3$,
  both of which are reduced. Thus, $\r(s\w_2s\w_3)$ is a positive
  root, and this establishes the base case.

  For the inductive step, we need to show that $\r(\w_1s\w_2\cdots
  s\w_n)$ is positive. By linearity,
  \begin{align*}
    \r(\w_1s\w_2s\w_3\cdots s\w_n) &=F_{\w_n}\circ F_s\circ\cdots\circ
    F_{\w_3}\circ F_s\circ F_{\w_2}[\r(\w_1)+c_1\r(s)]
    \\ &=\r(\w_1\w_2s\w_3\cdots s\w_n)+c_1\r(s\w_2s\w_3\cdots
    s\w_n)\,.
  \end{align*}
  The first root is positive by the induction hypothesis, so to prove
  the lemma, it suffices to show that $\r(s\w_2s\w_3\cdots s\w_n)$ is
  positive. Using~\eqref{eq:basecase}, we get
  \begin{align*}
 \r(s\w_2s\w_3\cdots s\w_n)&=F_{s\w_{4}\cdots s\w_n}[\r(s\w_2s\w_3)] \\
    &=[\r(t\u_1\cdots t\u_ks\w_3s\w_4\cdots s\w_n)+\cdots+\r(t\u_{k-1}t\u_ks\w_3s\w_4\cdots s\w_n)]\\
    &\quad +\r(t\u_k\w_3s\w_4\cdots s\w_n)\,.
  \end{align*}
  Each of these are roots of expressions that are subwords of either
  the word $\w_2s\w_3s\w_4\cdots s\w_n$ or of $\w_2\w_3s\w_4\cdots
  s\w_n$, both of which are reduced by the induction hypothesis. 
\end{proof}

\section{Logarithmic CFC elements}\label{sec:affineweyl}

Recall Theorem~\ref{thm:speyer}, which said that Coxeter elements are
logarithmic if and only if they are torsion-free. The following
theorem generalizes this to CFC elements without large bands.
\begin{theorem}
  \label{thm:largebands}
  Let $w$ be a CFC element of $W$ with no large bands. Then $w$ is
  logarithmic if and only if $w$ is torsion-free.
\end{theorem}
\begin{proof}
The forward direction is trivially handled by Proposition~\ref{prop:logarithmic}, so we will only consider the reverse direction. Moreover, it suffices to consider the case where $W$ is irreducible and $w$ has full support. This means that either $|S|\geq 3$, or $W$ is the free Coxeter group on $2$ generators (i.e., $m(s_1,s_2)=\infty$). The latter case is trivial and so we will ignore it and assume that $|S|\geq 3$.

Let $\w$ be a reduced expression for $w$. If $\w^{k}$ is FC for all $k$, then we are done.  Assume otherwise. By Theorem~\ref{thm:bands}, with the assumption that $w$ has no large bands, $w$ must have a small $st$-band for some $s,t\in S$, meaning the occurrences of $s$ and $t$ in $\w$ are both unique.  Assume without loss of generality that $s$ (and not $t$) is the endpoint, and let $W'$ be the parabolic subgroup of $W$ obtained by removing $s$. By Corollary~\ref{cor:remove-s}, deleting the unique occurrence of $s$ from $\w$ yields a reduced expression $\w'$ for a CFC element $w'$ of $W'$ that has no large bands. From here, we have two potential ways to show that $w$ is logarithmic. If $W'$ is infinite and $w'$ is a Coxeter element, then $w$ is a Coxeter element of $W$, and hence logarithmic by Theorem~\ref{thm:speyer}. Alternatively, if $w'$ is logarithmic, then $(\w')^k$ is reduced for all $k$, and so by Lemma~\ref{lemma:sinsert}, $\w^k$ is reduced as well.

We will proceed by induction on $|S|$.  For the base case, suppose
that $|S|=3$, meaning $W'$ is of type $I_2(m)$.  Since $t$ occurs
exactly once in $\w$, the remaining generator of $I_{2}(m)$ occurs
precisely once.  Thus, $w'$ is a Coxeter element, and we are done. 

For the inductive step, assume $|S|\geq 4$.  If $W'$ is infinite, then
by induction, $(\w')^k$ is reduced in $W'$, and so $w$ must be
logarithmic. Thus, suppose that $W'$ is finite. We have two cases. If
$W'$ has no large odd endpoints, then it follows from
Corollary~\ref{cor:CFCelements} that $w'$ is a Coxeter element.  Now,
suppose that $W'$ has a large odd endpoint. Since $W'$ is finite and
of rank at least $3$, it must be of type $H_3$ or $H_4$. In this case,
the only possibilities for the Coxeter graph of $W$ are shown in
Figure~\ref{fig:lastcase}.
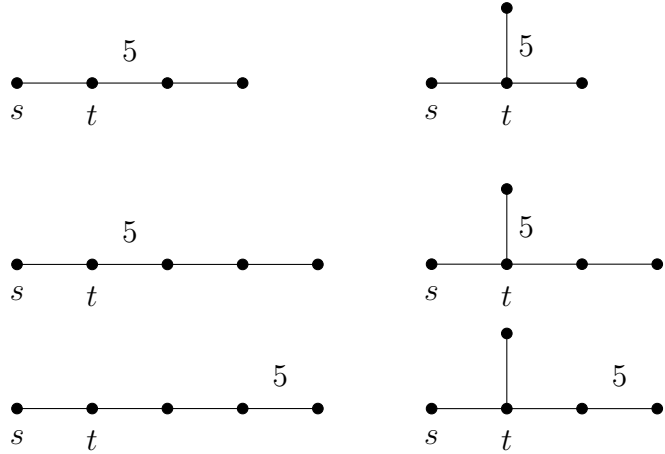
\begin{figure}[ht]
\centering
\begin{tabular}{ll}
 \begin{tikzpicture}[scale=1.0]
    \draw[fill=black] \foreach \x in {1,2,...,4} {(\x,3) circle (2pt)};
    \draw {(1,3) node[label=below:$s$]{}
      (2.5,3) node[label=above:$5$]{}
      (2,3) node[label=below:$t$]{}
      [-] (1,3) -- (4,3)};
  \end{tikzpicture}
&
\quad  \quad \begin{tikzpicture}[scale=1.0]
\draw[fill=black] \foreach \x in {1,2,3} {(\x,6.5) circle (2pt)};\draw[fill=black] (2,7.5) circle (2pt);
\draw{
(2,7) node[label=right:$\!\!5$]{}
[-] (1,6.5) -- (3,6.5)
[-] (2,6.5) -- (2,7.5)
(1,6.5) node[label=below:$s$]{}
(2,6.5) node[label=below:$t$]{}};
\end{tikzpicture}\\
\\
\begin{tikzpicture}[scale=1.0]
    \draw[fill=black] \foreach \x in {1,2,...,5} {(\x,3) circle (2pt)};
    \draw {(1,3) node[label=below:$s$]{}
      (2.5,3) node[label=above:$5$]{}
      (2,3) node[label=below:$t$]{}
      [-] (1,3) -- (5,3)};
  \end{tikzpicture}
&
\quad  \quad 
\begin{tikzpicture}[scale=1.0]
\draw[fill=black] \foreach \x in {1,2,3,4} {(\x,6.5) circle (2pt)};\draw[fill=black] (2,7.5) circle (2pt);
\draw{
(2,7) node[label=right:$\!\!5$]{}
[-] (1,6.5) -- (4,6.5)
[-] (2,6.5) -- (2,7.5)
(1,6.5) node[label=below:$s$]{}
(2,6.5) node[label=below:$t$]{}};
\end{tikzpicture}
\\
\begin{tikzpicture}[scale=1.0]
    \draw[fill=black] \foreach \x in {1,2,...,5} {(\x,3) circle (2pt)};
    \draw {(1,3) node[label=below:$s$]{}
      (4.5,3) node[label=above:$5$]{}
      (2,3) node[label=below:$t$]{}
      [-] (1,3) -- (5,3)};
  \end{tikzpicture}
&
\quad  \quad 
\begin{tikzpicture}[scale=1.0]
\draw[fill=black] \foreach \x in {1,2,3,4} {(\x,6.5) circle (2pt)};\draw[fill=black] (2,7.5) circle (2pt);
\draw{
(3.5,6.5) node[label=above:$5$]{}
[-] (1,6.5) -- (4,6.5)
[-] (2,6.5) -- (2,7.5)
(1,6.5) node[label=below:$s$]{}
(2,6.5) node[label=below:$t$]{}};
\end{tikzpicture}\\
\end{tabular}
\caption{The last remaining obstructions to Theorem~\ref{thm:largebands}.}
\label{fig:lastcase}
\end{figure}
For each of these six Coxeter graphs, we may assume that $s$ and $t$
are the indicated vertices. (Note that any other choice would result
in either an isomorphic copy of $W'$ or an infinite group.) These six
graphs fall into two cases. In the top four graphs, $t$ is involved in
a strength $5$ bond, and so the uniqueness of the occurrence of $t$
forces $w'$ to be a Coxeter element (of $H_3$ or $H_4$) because we
have $j=1$ in Lemma~\ref{lemma:cfch}(ii). In the bottom two graphs, $t$ is
not involved in a strength 5 bond, so $w'$ has a large band if and only
if $w$ does, and by Lemma~\ref{lemma:cfch}(iii), $w'$ is a Coxeter
element. In either case, it follows that $w$ is also a Coxeter
element, and hence $w$ is logarithmic. 
\end{proof}
\begin{cor}
  \label{cor:baby-RHD}
  Let $(W,S)$ be a Coxeter system without large odd endpoints. An
  element $w\in\CFC(W)$ is logarithmic if and only if it is
  torsion-free.
\end{cor}
\begin{proof}
  The forward direction is handled by
  Proposition~\ref{prop:logarithmic}. For the converse, let $w$ be
  torsion-free with reduced expression $\w$. We may assume it has full
  support and $W$ is irreducible. Since $W$ has no large odd
  endpoints, $w$ has no large bands, and hence is logarithmic by
  Theorem~\ref{thm:largebands}. 
\end{proof}

The class of Coxeter groups without large odd endpoints includes all
affine Weyl groups and simply-laced Coxeter groups. In fact, we can
say even more about CFC elements in affine Weyl groups. The following
corollary says that the only logarithmic CFC elements with bands in an
affine Weyl group are the Coxeter elements.

\begin{cor}
  \label{cor:affMUL2}
  Let $W$ be an affine Weyl group, and $\w$ a reduced expression for
  $w\in\CFC(W)$ with full support. Then $w$ is logarithmic and either
  \begin{enumerate}[label=(\roman*)]
  \item $w$ is a Coxeter element, or
  \item $\w^k\in\FC(W)$ for all $k \in \N$.
  \end{enumerate}
\end{cor}

\begin{proof}
  Since $W$ is an affine Weyl group, each $m(s, t)\in\{1, 2, 3, 4, 6,
  \infty\}$, which means that $W$ has no large odd endpoints, and none
  of its CFC elements have large bands. The proof of
  Theorem~\ref{thm:largebands} carries through, except that the only
  situation where (i) and (ii) do not occur is the case where it is
  possible to remove an element of $S$ and still be left with an
  infinite Coxeter group. The proof follows from a well-known (and
  easily checked) property of affine Weyl groups, which is that all of
  their proper parabolic subgroups are finite. 
\end{proof}

\begin{example}\label{ex:ae}
  Here are some examples of CFC elements in affine Weyl groups, and
  what our results tell us about their properties.
 \begin{enumerate}[label=(\roman*)]
  \item Consider the affine Weyl group of type $\widetilde{A}_n$,
    for $n\geq 2$.  The corresponding Coxeter graph is an $(n+1)$-gon,
    all of whose edges have bond strength three.  Let $c$ be a Coxeter
    element of $W(\widetilde{A}_n)$.  Then $c$ is CFC, and is
    logarithmic by Theorem~\ref{thm:speyer}.  Since $\widetilde{A}_n$
    has no endpoints and $c$ has full support, $c$ cannot have any
    bands. By Theorem~\ref{thm:bands}, $c^k$ is FC for all $k$, and
    now we can use Proposition~\ref{prop:w^k} to deduce that $c^k$ is CFC
    for all $k$.
  \item Consider the affine Weyl group of type
    $\widetilde{E}_8$, or in other words, type $E_9$, and let $c$ be a
    Coxeter element of $W(\widetilde{E}_8)$.  Again, by
    Theorem~\ref{thm:speyer}, $c$ is logarithmic.  However,
    $\widetilde{E}_8$ is FC-finite, so it cannot be the case that
    $c^k$ is FC (and hence CFC) for all $k$. By Lemma~\ref{lem:w^2},
    $c^2$ is not FC, and by Theorem~\ref{thm:bands}, $c$ must have a
    band. 
  \item Recall from Example~\ref{ex:ex1}\ref{ex:affine-E6} that
    $w=s_1s_3s_2s_4s_3s_5s_4s_6s_0s_3s_2s_6$ is a CFC element in the
    affine Weyl group of type $\widetilde{E}_6$.  Though the Coxeter
    graph has three odd endpoints, $w$ has no bands, which is easily
    verified from the observation that each generator adjacent to an
    endpoint occurs twice in $w$. By Theorem~\ref{thm:bands}, $w^k$ is
    FC for all $k$, and by Proposition~\ref{prop:w^k}, $w^k$ is CFC
    for all $k$.
  \item As in Example~\ref{ex:ex1}\ref{ex:affine-C4}, let
    $w_{1}=s_{0}s_{2}s_{4}s_{1}s_{3}$ and
    $w_{2}=s_{0}s_{1}s_{2}s_{3}s_{4}s_{3}s_{2}s_{1}$ be elements in
    $W(\widetilde{C}_{4})$. Since $w_{1}$ and $w_{2}$ are CFC elements
    with full support, by Corollary~\ref{cor:baby-RHD}, both are
    logarithmic.  Moreover, since $W(\widetilde{C}_4)$ has no odd
    endpoints, CFC elements with full support in
    $W(\widetilde{C}_{4})$ have no bands, so powers of $w_{1}$ and
    $w_{2}$ remain FC (Theorem~\ref{thm:bands}), and CFC
    (Proposition~\ref{prop:w^k}).

\end{enumerate}
\end{example}

\section{Conclusions and Future Work}\label{sec:conc}

Our motivation for defining and studying the CFC elements arose from
recent work on Coxeter elements described in Section~\ref{sec:cfc}, in
which the source-to-sink operation arose. It seemed that certain
properties of Coxeter elements were not due to the fact that every
generator appears once, but rather that conjugation is described
combinatorially by this source-to-sink operation. Thus, CFC elements
seemed like the natural generalization, because they are the largest
class of elements for which the source-to-sink operation
extends. Indeed, we showed that for any CFC element $w$ (without large
bands), $w$ is logarithmic iff $w$ is torsion-free. This generalizes
Speyer's recent result that says the same for the special case of
Coxeter elements. If the source-to-sink operation is indeed crucial to
this logarithmic property, then there should be a simple example of a
cyclically reduced non-CFC element that fails to be logarithmic. The
following example of this was pointed out recently by
M.~Dyer~\cite{Dyer:11}, where $W$ is the affine Weyl group
$\widetilde{C}_2$, and $w$ the following non-CFC element:
\[
  \centering
  \begin{tikzpicture}[scale=1.0]
    \draw[fill=black] \foreach \x in {1,2,3} {(\x,3) circle (2pt)};
    \draw {(0,3) node{$\widetilde{C}_2$}
      (1,3) node[label=below:$s_0$]{}
      (1.5,3) node[label=above:$4$]{}
      (2.5,3) node[label=above:$4$]{}
      (2,3) node[label=below:$s_1$]{}
      (3,3) node[label=below:$s_2$]{}
      [-] (1,3) -- (3,3)
      (6,3) node[label=right:${w=s_0s_1s_0s_1s_2.}$]{}
    };
  \end{tikzpicture}
\]
Clearly, $w$ is cyclically reduced and torsion-free, but
\[
w^2=(s_0s_1s_0s_1s_0)(s_2s_1s_0s_1s_2)=
(s_1s_0s_1s_0s_0)(s_2s_1s_0s_1s_2)=(s_1s_0s_1)(s_2s_1s_0s_1s_2),
\]
and so $\ell(w^2)<2\ell(w)$. Obviously, such a counterexample works
for any $m(s_1,s_2)\geq 4$. Thus, being cyclically reduced and
torsion-free together are \emph{not} sufficient for a non-CFC element
to be logarithmic. So, what are the necessary and sufficient
conditions for an arbitrary element in a Coxeter group to be
logarithmic? In this paper, we formalized the root automaton of a
Coxeter group in a new way, and it led to a new technique for proving
reducibility. We expect this approach to be useful for other questions
about reducibility. However, new geometric tools would need to be
developed to attack this general question for non-CFC
elements. In~\cite{Krammer:09}, D.~Krammer defines the ``axis'' of an
element, which generalizes the property of being logarithmic (which
Krammer calls \emph{straight}). Krammer proves some results on the
axis, but does not use these to draw conclusions about combinatorial
properties of logarithmic elements. We do not know yet whether these
techniques will help, but it remains a possibility.

Another natural question is whether torsion-free CFC elements with
large bands are necessarily logarithmic. Consider the following sets of
elements shown below.
\[
\left\{\begin{array}{c} \mbox{Coxeter} \\ \mbox{elements}
\end{array}\right\}\;\,\subset\,\;
\left\{\begin{array}{c}
  \mbox{CFC elements} \\
  \mbox{w/o large bands}
\end{array}\right\}\;\,\subset\,\;
\left\{\begin{array}{c}
  \mbox{CFC elements}
\end{array}\right\}\;\,\subset\,\;
\left\{\begin{array}{c}
  \mbox{cyclically reduced} \\
    \mbox{elements}
\end{array}\right\}
\]
The source-to-sink operation holds for these first three sets, but
breaks down for the fourth. Being torsion-free implies being
logarithmic for elements in the first two sets, but not for elements
in the fourth. Is it also sufficient for elements in the third set?
If so, that would imply that in any Coxeter group, a CFC element is
logarithmic if and only if it is torsion-free (recall that in
Corollary~\ref{cor:baby-RHD}, we proved that this is true for all Coxeter
groups without large odd endpoints), and this would give even more
evidence that the combinatorics behind the source-to-sink operation is
governing the logarithmic property. It is tempting to conjecture this
for purely aesthetic reasons, and it may in fact be true. However, we
do not have any firm mathematical evidence.

As mentioned earlier, we expect that these results will be useful in
better understanding the conjugacy problem in Coxeter groups. Since
the logarithmic property was key to establishing the cyclic version of
Matsumoto's theorem (as mentioned in the introduction) for Coxeter
elements, we expect that it will be necessary for CFC elements. We
conjecture that the cyclic version of Matsumoto's theorem holds for at
least the CFC elements (and likely much more), and once again, the
combinatorial techniques involving the source-to-sink operation should
play a central role. But does it hold for general torsion-free
cyclically reduced elements? If there is a counter-example, it is
certainly not obvious. In the meantime, progress towards this goal
should lead to valuable new developments in the combinatorial
understanding of reducibility and conjugacy. Understanding any
obstacles to this conjecture would also be of considerable interest,
and even if it were shown to be false, understanding when it fails
(and proving a modified version) would surely bring new
insight.


\subsection*{Acknowledgments}
  We thank Sara Billey, Hugh Denoncourt, Matthew Dyer, and Nathaniel
  Thiem for helpful comments.

\bibliographystyle{amsplain}

\end{document}